\documentclass[11pt,a4paper,dvipsnames]{article}
\usepackage[top=35mm,bottom=32mm,right=20mm,left=20mm]{geometry}
\usepackage[T1]{fontenc}
\usepackage[utf8]{inputenc}

\usepackage{mathptmx,lmodern,mathrsfs,amsthm,amsmath,graphicx,graphics,float,cases,xfrac,amssymb,bm,bbm,mathtools,gensymb,subcaption,tikz,tocloft,listings,pgffor,lipsum,microtype,titling,proof,fancyhdr,longtable,lscape,emptypage,parskip,booktabs,bm,stmaryrd,enumerate,autobreak,algorithm,algpseudocode,enumitem, wrapfig}
\usepackage{comment}
\usepackage[hidelinks]{hyperref}

\usepackage[
    noabbrev,   
    capitalise, 
    nameinlink,  
]{cleveref} 

\usepackage[immediate]{silence} 
\usepackage{sectsty}
\DeactivateWarningFilters[temp]

\usepackage[labelfont={bf,sf}]{caption}
\usepackage{pgfplots}
\usepgfplotslibrary{colormaps}
\usepackage[toc,page]{appendix}
\usepackage[%
maxnames=20,
giveninits=true,
isbn=false,
]{biblatex}
\usepackage{pgfplotstable}
\pgfplotsset{compat=1.7}
\usepgfplotslibrary{fillbetween}
\definecolor{color1}{RGB}{68,119,170}   
\definecolor{color2}{RGB}{102,204,238}  
\definecolor{color3}{RGB}{34,136,51}    
\definecolor{color4}{RGB}{50,200,100}    
\definecolor{color5}{RGB}{204,187,68}   
\definecolor{color6}{RGB}{221,204,119}  
\definecolor{color7}{RGB}{204,102,119}  
\definecolor{color8}{RGB}{136,34,85}    
\definecolor{color9}{RGB}{170,51,119}   
\definecolor{color10}{RGB}{102,102,102} 
\definecolor{color11}{RGB}{50,50,50} 
\definecolor{graph_color}{RGB}{204,102,119}

\pgfplotsset{
/pgfplots/colormap={mycolormap}{rgb=(0.4495,0.2869,0.5711) rgb=(0.8994,0.8907,0.9426)}
}

\DeclareMathAlphabet\mathbfcal{OMS}{cmsy}{b}{n}

\newcommand\norm[1]{\left\lVert#1\right\rVert}

\newcommand{\sym}{\mathbb{S}}

\newcommand\sgn{\operatorname{sgn}}

\DeclareMathOperator*{\argmax}{argmax}

\DeclareMathOperator*{\dom}{dom}
\DeclareMathOperator*{\Id}{Id}

\DeclareMathOperator*{\Cl}{cl}
\DeclareMathOperator*{\Int}{int}
\DeclareMathOperator*{\bdry}{bdry}

\newcommand*{\R}{\mathbb{R}}

\newcommand{\Rext}{\R \cup \{\infty\}}

\allsectionsfont{\normalfont\sffamily\bfseries}

\makeatletter
\def\th@plain{
  \thm@headfont{\normalfont\sffamily\bfseries}
  \itshape 
}

\def\th@definition{
  \thm@headfont{\normalfont\sffamily\bfseries}
  \thm@notefont{\normalfont\sffamily\bfseries}
}

\newtheoremstyle{myStyle1}
  {0.3cm}
  {0.3cm}
  {\itshape}
  {}
  {\normalfont\sffamily\bfseries}
  {:}
  {.5em}
  {}
\makeatother

\newtheoremstyle{myStyle2}
  {0.3cm}
  {0.3cm}
  {}
  {}
  {\normalfont\sffamily\bfseries}
  {:}
  {.5em}
  {}

\makeatother

\theoremstyle{myStyle1}
\newtheorem{thm}{Theorem}[section]
\newtheorem{prp}{Proposition}[section]
\newtheorem{cor}{Corollary}[section]

\newtheorem{lem}{Lemma}[section]

\theoremstyle{myStyle2}
\newtheorem{rem}{Remark}[section]
\newtheorem{exmp}{Example}[section]
  
\newcommand{\itemcref}[2]{\hyperref[#2]{\cref*{#1}\labelcref*{#2}}}

\crefname{equation}{}{}
\Crefname{equation}{}{}

\begin{document}
\title{\Large \sffamily\bfseries The Symmetry Coefficient of Positively Homogeneous Functions}

\author{Max Nilsson$^{\star}$ \and Pontus Giselsson$^{\star}$}
\date{%
    $^{\star}$Department of Automatic Control\\%
    Lund University, Lund, Sweden\\%
    \{\href{mailto:max.nilsson@control.lth.se}{max.nilsson}, \href{mailto:pontus.giselsson@control.lth.se}{pontus.giselsson}\}@control.lth.se\\[2ex]%
}
\maketitle
\begin{abstract}
The Bregman distance is a central tool in convex optimization, particularly in first-order gradient descent and proximal-based algorithms. Such methods enable optimization of functions without Lipschitz continuous gradients by leveraging the concept of relative smoothness, with respect to a reference function $h$. A key factor in determining the full range of allowed step sizes in Bregman schemes is the symmetry coefficient, $\alpha(h)$, of the reference function $h$. While some explicit values of $\alpha(h)$ have been determined for specific functions $h$, a general characterization has remained elusive. This paper explores two problems: (\textit{i}) deriving calculus rules for the symmetry coefficient and (\textit{ii}) computing $\alpha(\norm{\cdot}_2^p)$ for general $p$. We establish upper and lower bounds for the symmetry coefficient of sums of positively homogeneous Legendre functions and, under certain conditions, provide exact formulas for these sums. Furthermore, we demonstrate that $\alpha(\norm{\cdot}_2^p)$ is independent of dimension and propose an efficient algorithm for its computation. Additionally, we prove that $\alpha(\norm{\cdot}_2^p)$ asymptotically equals, and is lower bounded by, the function $1/(2p)$, offering a simpler upper bound for step sizes in Bregman schemes. Finally, we present closed-form computations for specific cases such as $p \in \{6,8,10\}$.
\end{abstract}
\section{Introduction} \label{sec::introduction}

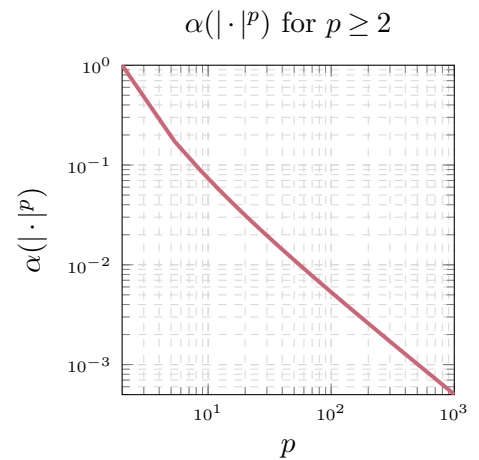
\begin{wrapfigure}{r}{0.35\linewidth}
        \begin{tikzpicture}
            \begin{loglogaxis}[
                width=\linewidth,
                height=\linewidth,
                xlabel={$p$},
                ylabel={$\alpha(|\cdot|^p)$},
                grid=both,
                grid style={dashed, gray!30},
                ticklabel style={font=\tiny},
                xmin=2,
                xmax=1000,
                ymax=1,
                ymin=5e-4,
                title={$\alpha(|\cdot|^p)$ for $p \geq 2$},
            ]
            \addplot+[
                color=graph_color,
                mark=none,
                line width=1.5pt,
                solid,
            ] table[
                x=p,
                y=alpha,
                col sep=comma,
            ] {./data/data.tex};
            \end{loglogaxis}
        \end{tikzpicture}
    \caption{Calculation of the symmetry coefficient of ${|\cdot|^p : \R \to \R}$ for $p \in [2, 10^3]$, utilizing \cref{alg:compute symmetry for 1d norm}.}
    \label{fig::plot_of_alpha_p}

\end{wrapfigure}

Since Bregman's seminal paper \cite{bregman1967relaxation}, the field of convex optimization has enjoyed the use of the non-Euclidean proximity measure called the \textit{Bregman distance}. For historical results on Bregman distances and their associated first-order gradient descent and proximal-based algorithms we refer the reader to \cite{nemirovskiĭ1983problem,beck2003mirror,censor1992proximal,chen1993convergence,eckstein1993nonlinear,de1986relaxed,bauschke2001joint,bauschke1997legendre,bauschke2003bregman}. These algorithms have a wide range of applications, such as in low-rank minimization \cite{dragomir2021quartic}, quadratic inverse problems \cite{bolte2018first}, $D$-optimal design problems \cite{lu2018relatively}, Fisher markets \cite{birnbaum2011distributed}, minimax problems \cite{teboulle1992entropic}, Poisson linear inverse problems \cite{bauschke2017descent}, and reinforcement learning \cite{sokota2022unified}.

An advantage of Bregman schemes---such as the NoLips algorithm in \cite{bauschke2017descent}---is their ability to handle functions that are differentiable but does not have a Lipschitz continuous gradient. The rectifying property is that of \textit{relative smoothness}, see \cite{birnbaum2011distributed,lu2018relatively,bauschke2017descent}. A problem's relative smoothness with respect to a reference function $h$, impacts the upper step size bound in the NoLips convergence analysis \cite{bauschke2017descent}. The step size bound also depends on the \textit{symmetry coefficient} $\alpha(h)$. The larger the symmetry coefficient $\alpha(h)$---the larger the step size upper bound in NoLips.

The Bregman distance is in general asymmetric and its worst case asymmetry is quantified in the symmetry coefficient. A Legendre function $h$ which does not enjoy an open effective domain, necessarily satisfies an absence of symmetry, i.e., $\alpha(h) = 0$, see \cite[Proposition 2]{bauschke2017descent}. The one-dimensional functions $|\cdot|^p : \R \to \R$ with $p=2$ and $p=4$ have symmetry coefficients $1$ and $2-\sqrt{3}$, respectively \cite{bauschke2017descent,teboulle2018simplified}. The result for $p=2$ is known to hold independently of the underlying dimension, i.e., the symmetry coefficient of $\norm{\cdot}_2^2 : \R^n \to \R$ equals 1 for all positive integers $n$, see \cite{bauschke2017descent}.

In this paper, we extend these results---concerning the computation of $\alpha(|\cdot|^p)$ and $\alpha(\norm{\cdot}_2^p)$ for $p > 1$---in several different ways, summarized here as part of our \textbf{main contributions:}
\begin{enumerate}[label={\bfseries C\arabic*}]
    \item In \cref{technical lemma,corollary to lemmas,alg:compute symmetry for 1d norm} we show that $\alpha(|\cdot|^p)$ can be computed with the bisection method over the unit interval. See \cref{fig::plot_of_alpha_p} for a plot of these values. \label{item::C_1}
    \item In \cref{p norm hard} we show that $\alpha(\norm{\cdot}_2^p) = \alpha(|\cdot|^p)$, independently of the underlying dimension $n \geq 1$ of $\norm{\cdot}_2^p : \R^n \to \R$. \label{item::C_2}
    \item In \cref{corassymptotic theorem higher dim} we prove that the function $\alpha(\norm{\cdot}_2^p) > 1/(2p)$ for $p > 2$. Furthermore, we prove the asymptotic equality $\alpha(\norm{\cdot}_2^p) \sim 1/(2p)$ as $p \to \infty$. See \cref{fig::plot_of_alpha_p} where this asymptotic bound is visible. \label{item::C_3}
\end{enumerate}

Since the step size upper bound in NoLips is an increasing function of the symmetry coefficient, the lower bound in \labelcref{item::C_3} can be used to give a valid relaxed upper step size bound for Legendre function of the form $h = \norm{\cdot}_2^p$. Moreover, this relaxed upper bound becomes tight when we pass to the limit $p \to \infty$, in the sense that $\alpha(\norm{\cdot}_2^p) \sim 1/(2p)$ as $p \to \infty$.

\begin{table}[t]
    \centering
    \begin{tabular}{|c||c|c|c|c|c|c|}
    \hline
         &  $h=|\cdot|^2$ & $h=|\cdot|^3$ & $h=|\cdot|^4$ & $h=|\cdot|^6$ & $h = |\cdot|^8$ & $h = |\cdot|^{10}$\\
    \hline
    \hline
      $\alpha(h)$ & 1 &$\frac{1 - \sqrt{2}\sqrt[4]{3} + \sqrt{3}}{2}$ & $2-\sqrt{3}$ & $\frac{7-3\sqrt{5}}{2}$ & \cref{eq:alpha8} & \cref{eq:alpha10} \\
    \hline
    \end{tabular}
    \caption{Special cases of the symmetry coefficient for Legendre functions $h$. For a derivation of the computation of $\alpha(|\cdot|^p)$ when $p \in \{3, 6, 8, 10\}$, see \cref{sec:closed form}.}
    \label{tab::special_cases_of_symmetry_coefficient}
\end{table}

Furthermore, we provide explicit and closed form expressions of $\alpha(|\cdot|^p)$ when $p \in \{3, 6, 8, 10\}$. See \cref{tab::special_cases_of_symmetry_coefficient} and \cref{sec:closed form}. Here, we will also discuss why finding corresponding closed forms when $p$ is even and $p \geq 12$ might prove to be infeasible.

The consideration of sums of functions of the form $\norm{\cdot}_2^p : \R^n \to \R$ has proven fruitful both for theoretical results and for applications \cite{bolte2018first,zhang2021proximal}. For instance, in \cite[Proposition 2.1]{lu2018relatively}, it was proven that every twice-differentiable convex function $f$, which satisfies a mild growth condition on $\norm{\nabla^2 f}$, is relatively smooth to a sum of functions of the form $\norm{\cdot}_2^p$. Therefore, these functions can often be used as reference functions in the NoLips algorithm, which leads us to consider what $\alpha(\sum_{i = 1}^m \lambda_i \norm{\cdot}_2^{p_i})$ equals, where $\lambda_i > 0$ and $p_i > 1$. A lower bound of the symmetry coefficient of the function $\omega_{\beta, \gamma} = \frac{\beta}{4}\norm{\cdot}_2^4 + \frac{\gamma}{2}\norm{\cdot}_2^2$, where $\beta, \gamma > 0$, was determined in \cite[Proposition B.1]{zhang2021proximal} as $\alpha(\omega_{\beta, \gamma}) \geq \min\{\beta, \gamma\}/(5\cdot \max\{\beta, \gamma\})$.

In this paper, we not only improve said lower bound, but exactly compute $\alpha(\omega_{\beta, \gamma})$ for all $\beta, \gamma > 0$. Moreover, we completely characterize the symmetry coefficient of these types of functions, $\sum_{i = 1}^m \lambda_i \norm{\cdot}_2^{p_i}$, as summarized below in our continued \textbf{main contributions:}

\begin{enumerate}[label={\bfseries C\arabic*}]
    \setcounter{enumi}{3}
    \item In \cref{Summation of Positively Hom} we provide an upper and lower bound of $\alpha\left(\sum_{i = 1}^m h_i\right)$ when all $h_i$ are Legendre and positively homogeneous of distinct degrees $p_i > 1$. \label{item::C_4}
    \item In \cref{thm::last theorem} we prove that these upper and lower bounds coincide when all $h_i$ are of the form $h_i = \lambda_i\norm{\cdot}_{r_i}^{p_i}$ where $\lambda_i > 0$, $p_i > 1$, and $r_i \in \{2, p_i\}$. \label{item::C_5}
\end{enumerate}

Using the theory covered in \labelcref{item::C_2,item::C_5}, it follows as a special case that $\alpha(\omega_{\beta, \gamma}) = \alpha(|\cdot|^4) = 2-\sqrt{3}$, see \cref{tab::special_cases_of_symmetry_coefficient}.

\subsection{Outline}

The outline of the paper is as follows. In \cref{preliminaries}, we assume that each $h_i$ is Legendre and in \cref{General Symmetry Lower Bound} we derive the lower bound of the symmetry coefficient referenced in \labelcref{item::C_4}. In \cref{sec::pos hom fun}, we consider the case when each $h_i$ is positively homogeneous of degree $p_i >  1$. For instance, when the functions have unique degree of positive homogeneity, we can in \cref{Summation of Positively Hom} derive the upper bound referenced in \labelcref{item::C_4}. Moreover, in \cref{prop::perfect_symmetry} we prove that $\alpha(h) = 1$ occurs if and only if $h$ is a strictly convex quadratic with full domain. This proof is different, and assumes weaker assumptions, than that found in \cite{bauschke2001joint}. In \cref{sym 1d} we consider the special case where $h_i =|\cdot|^{p_i} : \R \to \R$, with $p_i > 1$ for all $i \in \{1, \dots, m\}$. In \cref{Monotonically Decreasing Symmetry Coefficient} we show that the function $p \mapsto \alpha\left(\norm{\cdot}_2^p\right)$ is strictly monotonically decreasing on the domain $[2, \infty)$. This fact, together with the lower and upper bounds given by \cref{General Symmetry Lower Bound} and \cref{Summation of Positively Hom}, respectively, gives \labelcref{item::C_5}. Moreover, we show \labelcref{item::C_1} and \cref{thm:assymptotic theorem}, which is the main result behind \labelcref{item::C_3}. Finally, in \cref{sym nd}, we consider the case where $h_i = \norm{\cdot}_2^{p_i} : \R^n \to \R$ with $p_i > 1$ for all $i \in \{1, \dots, m\}$ for arbitrary $n \geq 1$. In particular, we show in \cref{p norm hard} the result referenced in \labelcref{item::C_2}, i.e., that $\alpha\left(\norm{\cdot}_2^p\right)$ is independent of the dimension of the underlying space.

\subsection{Notation and Definitions}\label{sec::notation_and_definitions}

The notation we use is standard and follows for example \cite{rockafellar1970convex} or \cite{bauschke2011convex}. 

Let $\langle \cdot, \cdot \rangle : \R^n \times \R^n \to \R$ denote the scalar product on $\R^n$, given by $\langle x, y \rangle = \sum_{i=1}^n x_iy_i$ for all $x, y \in \R^n$. Let $\mathbb{S}^n_{++}$ denote the set of symmetric positive definite matrices of size $n \times n$.  

Let $p \in \R_{++}$, $S \subset \R^n$ and $Y$ equal either $\R \cup \{\infty\}$ or $\R^n$. A function $f : \R^n \to Y$ is said to be \textit{positively homogeneous of degree} $p$ \textit{over} $S$ if $f(\lambda x) = \lambda^p f(x)$ holds for all $\lambda \in \R_{++}$ and all $x \in S$ such that $\lambda x \in S$. When $S = \R^n$, the function $f$ is simply said to be \textit{positively homogeneous of degree} $p$. 
 
 Given a function $h : \R^n \to \Rext$, the \textit{effective domain} of $h$ is the set $\dom h \coloneqq \{x \in \R^n \mid h(x) < \infty\}$. A proper and convex function $h$ is said to be \textit{essentially smooth} \cite[Section 26]{rockafellar1970convex} if the following three properties hold:
 \begin{enumerate}[label=(\roman*)]
    \item $S \coloneq \Int(\dom h)$ is non-empty,
    \item $h$ is differentiable on $S$,
    \item for all sequences $\left\{x^{(i)}\right\}_{i = 0}^\infty$ in $S$ converging to a point of $\bdry S$, then $\lim_{i \to \infty} \norm{\nabla h\left(x^{(i)}\right)}_2 = \infty$.
\end{enumerate} Equivalently, if $h$ is proper, closed and convex, then $h$ is essentially smooth if and only if $\partial h$ is a single-valued mapping \cite[Theorem 26.1]{rockafellar1970convex}, i.e., if $\partial h(x)$ contains at most one element for each $x \in \R^n$. The function $h$ is said to be a \textit{Legendre function} \cite[Section 26]{rockafellar1970convex} if $h$ is a proper, closed, and convex function that is strictly convex on $\Int(\dom h )$ and essentially smooth. Recall that, by \cite[Theorem 26.5]{rockafellar1970convex}, $h$ is a Legendre function if and only if $h^*$ is a Legendre function. 

Given a function $h : \R^n \to \Rext$ that is differentiable on $\Int(\dom h)$, the \textit{Bregman distance} $D_h : \R^n \times \R^n \to \Rext$ is defined as
\begin{equation*}
    D_h(x, y) \coloneq \begin{cases}
       h(x) - h(y) - \langle\nabla h(y), x - y \rangle & \text{if }y \in  \Int (\dom h),\\  
       \infty & \text{otherwise,} 
    \end{cases}
\end{equation*}
for all $(x, y) \in \R^n \times \R^n$. Note that when $h$ is a Legendre function, then $D_h(x, y) > 0$ holds for all $x, y \in \Int(\dom h)$ such that $x \neq y$, and if $D_h(x, y) = 0$ holds for some $x, y \in \Int(\dom h)$, then $x = y$. 

The \textit{symmetry coefficient} of a Legendre function $h : \R^n \to \Rext$ is defined as \begin{equation}\label{eq::definition_of_symmetry_coefficient}\alpha(h) \coloneq \inf\left\{\left. \frac{D_h(x, y)}{D_h(y, x)} \right| x, y \in \Int(\dom h)\text{ and } x\neq y \right \}.\end{equation} Since the symmetry coefficient is only defined for Legendre functions, essential smoothness implies that there exists points $x, y \in \Int(\dom h)$ such that $x\neq y$, and so the infimum in \eqref{eq::definition_of_symmetry_coefficient} ranges over a non-empty set. Furthermore, the strict convexity of $h$ over $\Int(\dom h)$ implies that both the numerator and denominator in \eqref{eq::definition_of_symmetry_coefficient} are always strictly positive. 

We will refer to the property $\alpha(h) = 1$ as $h$ having \textit{perfect symmetry}, following \cite{bauschke2017descent}.
\section{Preliminaries} \label{preliminaries}
The following proposition collects some basic operations under which the symmetry coefficient is preserved. This result will be used multiple times throughout this paper. Note that the property $\alpha(h) = \alpha(h^*) \in [0, 1]$ was shown in \cite{bauschke2017descent}.

\begin{prp} \label{basic proposition}
    Let $h : \R^n \to \R \cup \{\infty\}$ be a Legendre function and let $b \in \R^n$, $c \in \R$, $x_0 \in \R^n$, $\lambda \in \R_{++}$ and $L \in \R^{n \times n}$ be nonsingular. Then, \begin{equation}\label{basic proposition eq}\alpha(h) = \alpha(h^*) = \alpha(h + \langle b, \cdot\rangle + c) = \alpha(h(\cdot - x_0)) =\alpha(\lambda h) = \alpha(h \circ L) \in [0, 1]. \end{equation}
\end{prp}
\begin{proof}
    For the first equality, see \cite[Section 2.3]{bauschke2017descent}. For the rest of the equalities, let us define $\hat{h} : \R^n \to \R \cup \{\infty\}$ by \begin{equation}\label{eq::invariance}
    \hat{h}(x) \coloneq \lambda h(Lx-x_0) + \langle b, x \rangle + c
\end{equation} for all $x \in \R^n$. Since $L$ is nonsingular, the mapping $\phi : \Int (\dom \hat{h}) \to \Int (\dom h)$, defined by $\phi(x) \coloneqq Lx - x_0$ for all $x \in \Int (\dom \hat{h})$, is bijective. Since $h$ is a Legendre function, we get that $\hat{h}$ is a proper, closed, and convex function that is strictly convex on $\Int(\dom \hat{h})$. Moreover, since $L$ is nonsingular, we get from subdifferential calculus that $\partial {\hat{h}}(x) = \lambda L^\top \partial \hat{h}(Lx - x_0) + b$ holds for all $x \in \R^n$. Therefore, $\partial \hat{h}$ is a single-valued mapping, which in turn implies that $\hat{h}$ is essentially smooth \cite[Theorem 26.1]{rockafellar1970convex}. All in all, $\hat{h}$ is a Legendre function.  Furthermore, \begin{align*}
    D_{\hat{h}}(x, y) &= \hat{h}(x) - \hat{h}(y) - \langle \nabla \hat{h}(y), x-y \rangle \\
    &= \lambda h(Lx - x_0) - \lambda h(Ly - x_0) + \langle b, x-y\rangle - \langle \lambda L^\top \nabla h(Ly - x_0) + b, x-y\rangle \\
    &= \lambda \bigl( h(Lx - x_0) - h(Ly - x_0) - \langle \nabla h(Ly - x_0) , (Lx-x_0) - (Ly - x_0)\rangle\bigr) \\
    &= \lambda D_h(\phi(x), \phi(y))
\end{align*} holds for all $x, y \in \Int(\dom \hat{h})$. Therefore, $\alpha(\hat{h}) = \alpha(h)$ and the rest of the equalities in \eqref{basic proposition eq} hold. Finally, the inclusion in \eqref{basic proposition eq} follows since $\min\{D_h(x, y)/D_h(y, x), D_h(y, x)/D_h(x, y) \} \in (0, 1]$ holds for all $x, y \in \Int(\dom h)$ such that $x \neq y$.
\end{proof} 

The following proposition provides a general lower bound of the symmetry coefficient of $h \coloneqq \sum_{i=1}^m h_i$, as referenced in \labelcref{item::C_4}. It is general in the sense that we only assume that the functions $h_i$ are Legendre. The lower bound, given by \eqref{lowerbound}, will be central in proving \cref{Summation of Positively Hom}, \cref{prp::answering-proposition to question} and \cref{thm::last theorem}. In the latter two results, the symmetry coefficient will actually coincide with this lower bound: $\min_{i \in \{1, \dots, m\}}\left\{\alpha(h_i)\right\}$.

\begin{prp}\label{General Symmetry Lower Bound}
    Let $h_i : \R^n \to \Rext$ be a Legendre function for each $i \in \{1, \dots, m\}$ and suppose that $\cap_{i = 1}^m \Int(\dom h_i)$ is non-empty. Then, $h \coloneqq \sum_{i = 1}^m h_i$ is a Legendre function and \begin{equation}\label{lowerbound}
    \alpha(h) \in \left[ \min_{i \in \{1, \dots, m\}}\left\{\alpha(h_i)\right\}, 1\right].\end{equation}
\end{prp}
\begin{proof}
Since $\cap_{i = 1}^m \Int(\dom h_i)$ is non-empty, \cite[Theorem 6.5]{rockafellar1970convex} implies that $\Int (\dom h) = \cap_{i = 1}^m \Int(\dom h_i)$, after identifying the relative interior of $\dom h$ as simply the interior of $\dom h$. Therefore, $h$ is proper, closed, and strictly convex on $\Int (\dom h)$. By \cite[Theorem 23.8]{rockafellar1970convex}, $\partial h = \sum_{i = 1}^m \partial h_i$ and so $\partial h$ is a single-valued mapping, which implies that $h$ is essentially smooth \cite[Theorem 26.1]{rockafellar1970convex}. Therefore, $h$ is a Legendre function.

Let $x, y \in \Int(\dom h) = \cap_{i = 1}^m \Int(\dom h_i)$ such that $x \neq y$, then $$\frac{D_h(x, y)}{D_h(y, x)} = \frac{\sum_{i = 1}^m D_{h_i}(x, y)}{\sum_{i = 1}^m D_{h_i}(y, x)}.$$ For each $j \in \{1, \dots, m\}$ we have, by definition, that $D_{h_j}(x, y)/D_{h_j}(y, x) \geq \alpha(h_j) \geq \min_{i \in \{1, \dots, m\}}\left\{\alpha(h_i)\right\}$ and so $$\sum_{j = 1}^m D_{h_j}(x, y) \geq \sum_{j = 1}^m D_{h_j}(y, x) \min_{i \in \{ 1, \dots, m\}}\left\{\alpha(h_i)\right\}.$$ By \cref{basic proposition}, $\alpha(h) \in [0, 1]$, therefore \eqref{lowerbound} holds.
\end{proof}

If $h$ is separable, i.e., if $h(x) = \sum_{i = 1}^m h_i(x_i)$ holds for all $x = (x_1, \dots, x_m) \in \R^m$, then \cref{General Symmetry Lower Bound} can be refined.

\begin{prp}\label{Symmetry when Separable}
    Let $h_i:\R \to \Rext$ be a Legendre function for each $i \in \{ 1, \dots, m\}$. Let $h: \R^m \to \Rext$ be defined by $h(x) = \sum_{i=1}^m h_i(x_i)$ for all $x \in \R^m$. Then, $h$ is a Legendre function and $$\alpha(h) = \min_{i \in \{1, \dots, m\}} \{\alpha(h_i)\}.$$
\end{prp}
\begin{proof}
One could easily show that the function $h$ is proper, closed, and strictly convex on $\Int(\dom h) = \Pi_{i = 1}^m \Int(\dom h_i).$ Since $\partial h = \Pi_{i = 1}^m \partial h_i$, the mapping $\partial h$ is single-valued which implies that $h$ is essentially smooth \cite[Theorem 26.1]{rockafellar1970convex}. Therefore, $h$ is a Legendre function.

 Let $x, y \in \Int(\dom h)$ such that $x \neq y$. Let $\mathcal{J} \coloneq \{j \in \{1, 2, \dots, m\} \mid x_j \neq y_j\}$. By assumption, $\mathcal{J}$ is non-empty. For each $j \in \mathcal{J}$ we have that $D_{h_j}(x_j, y_j)/D_{h_j}(y_j, x_j) \geq \alpha(h_j) \geq \min_{i \in \{1, \dots, m\}}\left\{\alpha(h_i)\right\}$ and so $$\sum_{j \in \mathcal{J}} D_{h_j}(x_j, y_j) \geq \sum_{j \in \mathcal{J}} D_{h_j}(y_j, x_j) \min_{i \in \{ 1, \dots, m\}}\left\{\alpha(h_i)\right\}$$ which implies that
     \begin{align*}
        \frac{D_h(x, y)}{D_h(y, x)} &= \frac{\sum_{i = 1}^mD_{h_i}(x_i, y_i)}{\sum_{i = 1}^mD_{h_i}(y_i, x_i)} \\
        &= \frac{\sum_{j \in \mathcal{J}}D_{h_j}(x_j, y_j)}{\sum_{j \in \mathcal{J}}D_{h_j}(y_j, x_j)} \geq \min_{i \in \{1, \dots, m\}}\left\{\alpha(h_i)\right\}.
    \end{align*}

    It remains to show that $\alpha(h) \leq \min_{i \in \{1, \dots, m\}} \{\alpha(h_i)\}$. Let $j \in \{1,\dots, m\}$ be arbitrary. Let $x, y \in \Int(\dom h)$ such that $x_j \neq y_j$ and $x_i = y_i$ for all $i \in \{1, \dots, m\} \setminus \{j\}$. We then have that $$D_h(x, y) / D_h(y, x) = D_{h_j}(x_j, y_j)/D_{h_j}(y_j, x_j).$$ Therefore, $\alpha(h) \leq \alpha(h_j)$ and since $j$ was arbitrary we are done.
\end{proof}
\section{Positively Homogeneous Functions} \label{sec::pos hom fun}
In this section, we focus on a subclass of Legendre functions, specifically the class of Legendre and positively homogeneous functions. Here are this section's main results.
\begin{enumerate}[label={\bfseries R\arabic*}]
    \item If $h$ is Legendre and positive homogeneous of degree $p > 1$ over $\Int(\dom h)$ $\implies \dom h = \R^n$.
    
    A function $h : \R^n \to \R \cup \{\infty\} $ can be Legendre \textit{or} positively homogeneous of degree $p > 1$ over $\Int(\dom h)$, without necessarily having full effective domain. But if $h$ is both Legendre and positively homogeneous, then we can show that $\dom h = \R^n$. Furthermore, in that case, the functions $h, \nabla h,$ and $D_h$ can be shown to be positively homogeneous of degree $p, p-1,$ and $p$, respectively, see \cref{pos hom implies domain}. \label{R1}
    
    \item The function $h : \R^n \to \Rext$ is a strictly convex quadratic function with $\dom h = \R^n$ $\iff$ $\alpha(h) = 1$.

    Of particular interest is the left implication, since the right implication is a straightforward calculation. A variant of this result has been shown for the special case when $n = 1$, $\dom h = \R_{++}$, and $h$ assumed to be twice continuously differentiable on $\R_{++}$ in \cite{bauschke2001joint}. In our argument (see the proof of \cref{prop::perfect_symmetry}) we do not put any such extraneous assumptions on $h$. In fact, our argument assumes only that $h$ is a Legendre function.  \label{R2}
    
    \item If $h_i$ is positively homogeneous of degree $p_i > 1$ for all $i \in \{1, \dots, m\}$, where the degrees satisfy $p_1 < p_2 \leq \dots \leq p_{m-1} < p_m$, then $\alpha(h) \leq \min\{\alpha(h_1), \alpha(h_m)\}$.

    This will be proved, using \cref{pos hom implies domain}, in \cref{Summation of Positively Hom}. In the coming section, we will frequently combine this upper bound with the lower bound of \cref{General Symmetry Lower Bound}, in order to fully determine the symmetry coefficient for some special cases. \label{R3}
\end{enumerate}

\subsection*{The Domains of Positively Homogeneous Legendre Functions (R1)}
We turn to proving \labelcref{R1}, which will be done in \cref{cor::domains}. Before that, we provide two lemmas: \cref{lem::nabla h and D_h pos hom} and \cref{pos hom implies domain}. The former does not assume that the function $h$ is Legendre and describes how the positive homogeneity of $h$ over a set $S$ implies positive homogeneity of $\nabla h$ and $D_h$ over $S$ and $S \times S$, respectively. The latter does assume that $h$ is Legendre and positive homogeneous over $\Int(\dom h)$ and describes how this implies that $h$ must have full domain. These will be combined to show \cref{cor::domains}.

\begin{lem} \label{lem::nabla h and D_h pos hom}
    Let $p \in (1, \infty)$ and let $h : \R^n \to \R \cup \{\infty\}$ be differentiable on an open set $S \subset \R^n$. Suppose that $h$ is positively homogeneous of degree $p$ over $S$, then $\nabla h$ is positively homogeneous of degree $p-1$ over $S$ and $D_h$ is positively homogeneous of degree $p$ over $S \times S$.
\end{lem}
\begin{proof}
    If $S = \emptyset$ then the statement is vacuously true. Otherwise, let $x \in S$ and $\lambda \in \R_{++}$ be such that $\lambda x \in S$. The function $h \circ (\lambda \Id)$ is differentiable at $x$, and after differentiating we get $\lambda \nabla h(\lambda x) = \lambda^p \nabla h(x)$ which gives that $\nabla h$ is positively homogeneous of degree $p-1$ over $S$. 

    Let also $y \in S$ such that $\lambda y \in S$. Then \begin{align*}
        D_h(\lambda x, \lambda y) &= h(\lambda x) - h(\lambda y) - \langle \nabla h(\lambda y), \lambda x - \lambda y \rangle \\
        &= \lambda^p h(x) - \lambda^p h(y) - \lambda^p \langle \nabla h(y), x-y \rangle = \lambda^p D_h(x, y)
    \end{align*} and so $D_h$ is positively homogeneous of degree $p$ over $S \times S$.
    
\end{proof}
\begin{lem} \label{pos hom implies domain}
    Let $p \in (1, \infty)$ and $h : \R^n \to \R \cup \{\infty\}$ be a Legendre function. Suppose that $h$ is positively homogeneous of degree $p$ over $\Int(\dom h)$, then $\dom h = \R^n$.
\end{lem}
\begin{proof}
    Let $x \in \Int(\dom h)$ and define $t_0$ as the infimum over the non-empty interval $$C \coloneqq \{t \in \R_{++} \mid tx \in \Int(\dom h)\}.$$ Since $\Int(\dom h)$ is open, $t_0 < 1$. By \cite[Theorem 6.1]{rockafellar1970convex}, $\lambda x \in \Int(\dom h)$ holds for all $\lambda \in (t_0, 1)$ and so $t_0x \in \Cl(\dom h)$. By \cref{lem::nabla h and D_h pos hom}, $\norm{\nabla h(\lambda x)}_2 = \lambda^{p-1}\norm{\nabla h(x)}_2$ holds for all $\lambda \in (t_0, 1)$. Therefore, $\lim_{\lambda \downarrow t_0} \norm{\nabla h(\lambda x)}_2 \in \R$ and by the relative smoothness of $h$ we have that $t_0x \not\in \bdry(\Int(\dom h))$. However, since $t_0x \in \Cl (\dom h)$ we get that $t_0x \in \Int(\dom h)$. The only way this can happen is if $t_0 = 0$, since otherwise $t_0$ would not be the infimum over $C$. In a similar way, the supremum of $C$ is $\infty$. Therefore, if $x \in \Int(\dom h)$ then $\lambda x \in \Int(\dom h)$ holds for all $\lambda \in [0, \infty)$. In particular, we conclude that $0 \in \Int(\dom h)$.
    
    Now for any $y \in \R^n$ take $\lambda \in \R_{++}$ small enough such that $\lambda y \in \Int(\dom h)$. Therefore, $\lambda^{-1} \lambda y = y \in \Int(\dom h)$ and $\dom h = \R^n$.
\end{proof}
\begin{prp} \label{cor::domains}
    Let $p \in (1, \infty)$ and $h : \R^n \to \R \cup \{\infty\}$ be a Legendre function. Suppose that $h$ is positively homogeneous of degree $p$ over $\Int(\dom h)$. Then, $h$ is differentiable, and $h, \nabla h$, and $D_h$ are positively homogeneous of degree $p, p-1$, and $p$, respectively. 
\end{prp}
\begin{proof}
    From \cref{pos hom implies domain}, we get that $\dom h = \R^n$, and so $h$ is differentiable on $\Int(\dom h) = \R^n$. The conclusion now follows from letting $S = \R^n$ in \cref{lem::nabla h and D_h pos hom}.
\end{proof}

\subsection*{Perfect Symmetry (R2)}

We turn to proving \labelcref{R2}, which will be done in \cref{prop::perfect_symmetry}. Before that, let us take a moment to motivate why it was natural to assume that $p > 1$ in \cref{lem::nabla h and D_h pos hom}, \cref{pos hom implies domain} and \cref{cor::domains}. 

\begin{prp} \label{prp::covering cases of p}
    Let $p \in\R_{++}$ and let $h : \R^n \to \R \cup \{\infty\}$ be a Legendre function. Suppose that $h$ is positively homogeneous of degree $p$ over $\Int(\dom h)$. Then, $p \neq 1$. 
    Furthermore, \begin{enumerate}[label=(\roman*)]
        \item if $p \in (0, 1)$, then $\alpha(h) = 0$, \label{prp::cases of p and homogenuity (i)}
        \item if $p \in (1, 2)$, then there exists a Legendre function $\hat{h} : \R^n \to \R \cup \{\infty\}$ which is positively homogeneous of degree $q \in (2, \infty)$ such that $\alpha(h) = \alpha(\hat{h})$.  \label{prp::cases of p and homogenuity (ii)}
    \end{enumerate}
\end{prp}
\begin{proof}
    Suppose that $p = 1$. Let $x \in \Int(\dom h)$ and take $\lambda \in \R_{++} \setminus \{1\}$ such that $\lambda x \in \Int(\dom h)$. This can be done since $\Int(\dom h)$ is non-empty and open. Then, for any $\theta \in (0, 1)$ we get that \begin{equation} \label{eq::contradicting strict convexity}
        h((1-\theta)x + \theta \lambda x) = h((1-\theta + \theta \lambda) x) = (1-\theta)h(x) + \theta h(\lambda x),
    \end{equation} where we in the second equality used that $1-\theta + \theta \lambda > 0$ and $h$ is positively homogeneous of degree $1$. Clearly, \eqref{eq::contradicting strict convexity} contradicts the strict convexity of $h$ over $\Int(\dom h)$ and so $p\neq1$ follows.

    \begin{enumerate}[labelindent=0pt,wide=0pt]
        \item[\labelcref{prp::cases of p and homogenuity (i)}] Suppose now that $p \in (0, 1)$. Let $x \in \Int(\dom h)$ and follow the argument as in \cref{pos hom implies domain} to get that $\lambda x \in \Int(\dom h)$ for all $\lambda \in \R_{++}$. The only difference is that $\lim_{\lambda \downarrow 0} \norm{\nabla h(\lambda x)}_2 =\lim_{\lambda \downarrow 0} \lambda^{p-1}\norm{\nabla h(x)}_2 = \infty$ and so $0 \notin \Int(\dom h)$. Since $h$ is closed, $h(0) = \liminf_{y \to 0}h(y) \leq \lim_{\lambda \downarrow 0} h(\lambda x) = \lim_{\lambda \downarrow 0} \lambda^ph(x) = 0$ and so $0 \in \dom h$. In particular, this implies that $\dom h$ is not open and by \cite[Proposition 2]{bauschke2017descent} we have that $\alpha(h) = 0$.
        \item[\labelcref{prp::cases of p and homogenuity (ii)}] By \cref{pos hom implies domain}, we have that $\dom h = \R^n$. By \cite[Theorem 26.5, Corollary 15.3.1]{rockafellar1970convex}, the conjugate function $h^*$ is Legendre and is positively homogeneous of degree $q = p/(p-1) \in (2, \infty)$, therefore the statement holds with $\hat{h} \coloneqq h^*$.
    \end{enumerate}
\end{proof}

As \cref{prp::covering cases of p} shows, we can often assume that the degree of positive homogeneity $p$ is greater than or equal to $2$. For instance, if $p \in (1, 2)$ then when analyzing the symmetry coefficient of $\frac{1}{p} |\cdot|^p : \R \to \R$, we can just as well analyze the symmetry coefficient of $\frac{1}{q}|\cdot|^q : \R \to \R$, where $1/p + 1/q = 1$ and $q \in (2, \infty)$, since $\left(\frac{1}{p} |\cdot|^p\right)^* = \frac{1}{q} |\cdot|^q$ \cite[Section 12]{rockafellar1970convex}. This is precisely what we will do in \cref{sym 1d,sym nd}.

Let us now concentrate on the case $p = 2$, covered in  \cref{prop::perfect_symmetry}. First, let us recall in \cref{thm::converse of Euler} the converse of Euler's Theorem for homogeneous functions. This theorem is classical, but we include a proof for the sake of completion. The proof idea comes from \cite[Exercise 9.3.9]{apostol2000calculus}.

\begin{thm}[Converse of Euler’s Theorem for Homogeneous Functions] \label{thm::converse of Euler}
    Let $p \in (1, \infty)$ and $h : \R^n \to \R \cup \{\infty\}$ be differentiable on $\Int(\dom h)$. Suppose that $\Int(\dom h)$ is a convex set and that  $p h(x) = \langle \nabla h(x), x \rangle$ holds for all $x \in \Int(\dom h)$, then $h$ is positively homogeneous of degree $p$ over $\Int(\dom h)$.
\end{thm}
\begin{proof}
     If $\Int(\dom h) = \emptyset$ then the statement is vacuously true. Otherwise, fix $x \in \Int(\dom h)$. Define $\phi : (t_0, t_1) \to \R$ by $\phi(t) = h(tx) - t^p h(x)$ for all $t \in (t_0, t_1)$, where $t_0$ and $t_1$ is the infimum and supremum of the non-empty set $ S \coloneqq \{t \in \R_{++} \mid tx \in \Int(\dom h)\}$, respectively. Since $\Int(\dom h)$ is open, $t_0 < t_1$ and $1 \in (t_0, t_1)$. Since $S$ is a convex set, $(t_0, t_1) \subset S$. Therefore, $\phi$ is differentiable on $(t_0, t_1)$, and \begin{align*}
        \phi'(t) &= \langle \nabla h(tx), x \rangle - pt^{p-1} h(x) \\ 
        &= \frac{1}{t} \langle \nabla h(tx), tx \rangle - pt^{p-1} h(x) \\
        &= \frac{1}{t} p h(tx) - pt^{p-1} h(x) = \frac{p}{t} \phi(t)
    \end{align*} holds for all $t \in (t_0, t_1)$. This differential equation has the solution $\phi(t) = C t^{p}$, for any $C \in \R$. The condition $\phi(1) = 0$ gives that $C = 0$ and so it follows that $h$ is positively homogeneous of degree $p$ over $\Int(\dom h)$.
\end{proof}

\begin{prp} \label{prop::perfect_symmetry}
Let $h: \R^n \to \Rext$ be a Legendre function. Then $\alpha(h) = 1$ if and only if $h$ is a strictly convex quadratic function with $\dom h = \R^n$.
\end{prp}

\begin{proof}
    Suppose that $h$ is a strictly convex quadratic with $\dom h = \R^n$, i.e., that there exists some $Q \in \sym^n_{++}$, $b \in \R^n$ and $c \in \R$ such that
    $$h(x) = \frac{1}{2}\langle x, Qx \rangle + \langle b, x \rangle + c$$ holds for all $x \in \R^n$. By \cref{basic proposition}, the symmetry coefficient of $h$ equals the symmetry coefficient of $\frac{1}{2}\norm{\cdot}_2^2 : \R^n \to \R$, which has the associated Bregman distance \begin{align*}
    D_{\frac{1}{2}\norm{\cdot}_2^2}(x, y) &= \frac{1}{2}\norm{x}_2^2 - \frac{1}{2}\norm{y}_2^2 - \langle y, x-y\rangle = \frac{1}{2} \norm{x-y}_2^2
    \end{align*} for all $x, y \in \R^n$. Since $D_{\frac{1}{2}\norm{\cdot}_2^2}$ is symmetric about its arguments, we get from \eqref{eq::definition_of_symmetry_coefficient} that $\alpha\left(\frac{1}{2}\norm{\cdot}_2^2\right) = 1$.

     Now, suppose that $\alpha(h) = 1$. By \cref{basic proposition}, we may without loss of generality suppose that $0 \in \Int \dom h$, $h(0) = 0$, and $\nabla h(0) = 0$. We aim to show that there exists some $Q \in \sym^n_{++}$ such that $h(x) = \frac{1}{2}\langle x, Qx \rangle$ holds for all $x \in \R^n$. By assumption, \begin{equation}\label{eq::euler_homogeneous}
     D_h(x, y) = D_h(y, x)
     \end{equation} holds for all $x, y \in \Int(\dom h)$. We get that \begin{align*}
    2h(x) -  \langle \nabla h(x), x \rangle &= \left(h(x) - h(0) - \langle \nabla h(0), x\rangle\right) - \left( h(0) - h(x) - \langle \nabla h(x), -x\rangle \right)\\
    &= D_h(x, 0) - D_h(0, x) = 0
    \end{align*} holds for all $x \in \Int(\dom h)$, where we in the last equality used \eqref{eq::euler_homogeneous} for $y = 0$. Therefore, 
    \begin{equation}\label{eq::poshom2}
        h(x) = \frac{1}{2}\langle \nabla h(x), x \rangle
    \end{equation} holds for all $x \in \Int(\dom h)$. By \cref{thm::converse of Euler}, $h$ is positively homogeneous of degree 2 over the set $\Int(\dom h)$. By \cref{pos hom implies domain}, $\dom h = \R^n$. Therefore, \begin{align*}
        \langle \nabla h(x), y \rangle  -\langle \nabla h(y) , x \rangle &= \left(\frac{1}{2}\langle \nabla h(x), x\rangle - \frac{1}{2}\langle \nabla h(y), y\rangle - \langle \nabla h(y), x-y \rangle\right) \\ &\quad - \left(\frac{1}{2}\langle \nabla h(y), y\rangle - \frac{1}{2}\langle \nabla h(x), x\rangle - \langle \nabla h(x), y-x \rangle\right) \\
        & = \left(h(x) - h(y) - \langle \nabla h(y), x-y \rangle\right) - \left(h(y) - h(x) - \langle \nabla h(x), y-x \rangle\right) \\
        &= D_h(x, y) - D_h(y, x) = 0
    \end{align*}  holds for all $x, y \in \R^n$, where we in the second equality used \eqref{eq::poshom2} and in the last equality used \eqref{eq::euler_homogeneous}. Therefore, 
    \begin{equation}\label{eq::poshom2 identity}
        \langle \nabla h(x), y \rangle = \langle \nabla h(y) , x \rangle
    \end{equation} holds for all $x, y \in \R^n$. In particular, $$\langle \nabla h(\lambda_1 x_1 + \lambda_2 x_2), y \rangle = \langle \lambda_1 x_1 + \lambda_2 x_2, \nabla h(y) \rangle = \langle\lambda_1  \nabla h(x_1) + \lambda_2 \nabla h(x_2), y\rangle$$ holds for all $x_1, x_2, y \in \R^n$ and $\lambda_1, \lambda_2 \in \R$. Setting $$y = \nabla h(\lambda_1 x_1 + \lambda_2 x_2)-(\lambda_1  \nabla h(x_1) + \lambda_2 \nabla h(x_2))$$ gives that $$\langle \nabla h(\lambda_1 x_1 + \lambda_2 x_2) - (\lambda_1  \nabla h(x_1) + \lambda_2 \nabla h(x_2)), y \rangle = \norm{y}_2^2 = 0,$$ i.e., $\nabla h$ is linear. Since $\nabla h$ is also continuous \cite[Corollary 17.43]{bauschke2011convex} we have that $h(x) = \frac{1}{2}\langle x, Qx\rangle$ for some $Q \in \mathbb{S}^{n}$ \cite[Proposition 2.58]{bauschke2011convex}. Since $h$ is strictly convex, ${Q \in \mathbb{S}_{++}^n.}$
\end{proof}

\subsection*{The Symmetry Coefficient of Sums of Positively Homogeneous Functions (R3)}

Finally, we turn to proving \labelcref{R3}.

\begin{prp} \label{Summation of Positively Hom}
    Let the function $h_i : \R^n \to \R\cup \{\infty\}$ be Legendre and positively homogeneous of degree $p_i \in (1, \infty)$ over $\Int(\dom h_i)$, for each $i \in \{ 1, \dots, m\}$. Suppose that $\{p_i\}_{i = 1}^m$ is a sequence satisfying $p_1 < p_2 \leq \dots \leq p_{m-1} < p_m$. Then, $h = \sum_{i = 1}^m h_i$ is a Legendre function and \begin{equation} \label{eq:pos_hom_bound}
        \alpha(h) \in \left[\min_{i \in\{1, \dots, m\}} \{ \alpha(h_i)\}, \ \min\{ \alpha(h_1), \alpha(h_m)\}\right].
    \end{equation}
\end{prp}
\begin{proof}
    By \cref{cor::domains}, $\dom h_i = \R^n$ holds for each $i \in \{1, \dots, m \}$ and so $h$ has full domain and is a Legendre function. Let $x, y \in \R^n$ be such that $x \neq y$, then 
    \begin{align} \label{eq::summation of positively hom 1}
        \lim_{\lambda \downarrow 0}\frac{D_h(\lambda x, \lambda y)}{D_h(\lambda y, \lambda x)} &= \lim_{\lambda \downarrow 0}\frac{\sum_{i = 1}^m \lambda^{p_i}D_{h_i}(x, y)}{\sum_{i = 1}^m \lambda^{p_i}D_{h_i}(y, x)} \nonumber \\
        &=\lim_{\lambda \downarrow 0} \frac{D_{h_1}(x, y) + \sum_{i = 2}^m \lambda^{p_i-p_1}D_{h_i}(x, y)}{D_{h_1}(y, x) + \sum_{i = 2}^m \lambda^{p_i-p_1}D_{h_i}(y, x)} = \frac{D_{h_1}(x, y)}{D_{h_1}(y, x)},
    \end{align} where the first equality follows from \cref{cor::domains} and the third equality follows from the fact that $p_i - p_1 > 0$ holds for all $i \in \{1, \dots, m\}$. Similarly, \begin{equation} \label{eq::summation of positively hom 2}
        \lim_{\lambda \to \infty}\frac{D_h(\lambda x, \lambda y)}{D_h(\lambda y, \lambda x)} 
        =\lim_{\lambda \to \infty} \frac{D_{h_m}(x, y) + \sum_{i = 1}^{m-1} \lambda^{p_i-p_m}D_{h_i}(x, y)}{D_{h_m}(y, x) + \sum_{i = 1}^{m-1} \lambda^{p_i-p_m}D_{h_i}(y, x)} = \frac{D_{h_m}(x, y)}{D_{h_m}(y, x)}.
    \end{equation} Combining \eqref{eq::summation of positively hom 1} and \eqref{eq::summation of positively hom 2} gives that $\alpha(h) \leq \min\{\alpha(h_1), \alpha(h_m)\}$. The lower bound of $\alpha(h)$ in \eqref{eq:pos_hom_bound} is given by \cref{Summation of Positively Hom}.
\end{proof}
\begin{cor} \label{cor::2 functions}
    Let the functions $h_1, h_2 : \R^n \to \R \cup \{\infty\}$ be Legendre and positively homogeneous of degree $p_1 \in (1, \infty)$ and $p_2 \in (1, \infty)$ over $\Int(\dom h_1)$ and $\Int(\dom h_2)$, respectively. Suppose that $p_1 \neq p_2$. Then, $h_1 + h_2$ is a Legendre function and $$\alpha(h_1 + h_2) = \min\{\alpha(h_1), \alpha(h_2)\}.$$
\end{cor}
\begin{proof}
    Apply \cref{Summation of Positively Hom} with $m = 2$.
\end{proof}
\begin{rem} \label{rem::necessity of strict degrees}
    We will now argue the necessity of the condition $p_1 \neq p_2$ in \cref{cor::2 functions}. Let $a, b \in \R_{++}$ and consider the function $h_{a, b} : \R \to \R$ defined by $$h_{a, b}(x) = \begin{cases}
ax^2 &\text{ if } x \geq 0,\\
bx^2 &\text{ otherwise,}
\end{cases}$$ for all $x \in \R$.  The function $h_{a, b}$ is Legendre and positively homogeneous of degree 2. Since $h_{a, b}(x) = h_{b, a}(-x)$ holds for all $a, b \in \R_{++}$ and $x \in \R$, it follows from \cref{basic proposition} that $\alpha(h_{a, b}) = \alpha(h_{b, a})$. If $a \neq b$, then \cref{prop::perfect_symmetry} gives that $\alpha(h_{a, b}) < 1$ and also that $\alpha(h_{a, b} + h_{b, a}) = \alpha(h_{a+b, a+b}) = 1$. So letting $h_1 \coloneqq h_{a, b}$ and $h_2 \coloneqq h_{b, a}$ we have that both $h_1$ and $h_2$ are Legendre functions that are positively homogeneous of degree 2, but $$\alpha(h_1 + h_2) > \min\{\alpha(h_1), \alpha(h_2)\}.$$ In fact, it can be verified---with a similar method as will be introduced in the next section---that \begin{equation} \label{eq::alpha h_ab}
    \alpha(h_{a, b}) = \min \left\{ \frac{1 + \sqrt{a/b}}{1 + \sqrt{b/a}}, \frac{1 + \sqrt{b/a}}{1 + \sqrt{a/b}}\right\}
\end{equation} holds for all $a, b \in \R_{++}$.
\end{rem}

\begin{exmp}
 Under the hypothesis of \cref{Summation of Positively Hom} it is not in general possible to improve the interval in \eqref{eq:pos_hom_bound} to a point estimate, as illustrated by the following example.
 
 Let $h \coloneqq \frac{1}{p}|\cdot|^p : \R \to \R$ for $p = 4$. It was shown  in \cite[Section 2.3]{bauschke2017descent} that $h$ has symmetry coefficient $\alpha(h) = 2 - \sqrt{3}$. The conjugate function $h^*$ has by \cref{basic proposition} the same symmetry coefficient, $\alpha(h^*) = \alpha(h)$, while being positively homogeneous of degree $p/(p-1) = 4/3$ \cite[Section 12]{rockafellar1970convex}. Let $b \in \R_{++}$ and consider the function $h_{1, b}$, following the notation of \cref{rem::necessity of strict degrees}, which is positively homogeneous of degree $2$. From \eqref{eq::alpha h_ab} we get that $\lim_{b \downarrow 0} \alpha(h_{1, b}) = 0$ and the lower bound of \eqref{eq:pos_hom_bound} approaches 0. Furthermore, one can verify that when $b = 10^{-7}, x = 10^{-3}$ and $y = -5\cdot 10^{-2}$, then $$\frac{D_{h^* + h_{1, b} + h}( x,  y)}{D_{h^* + h_{1, b} + h}( y, x)} < 0.2676 <  \min\{ \alpha(h^*), \alpha(h)\} = 2 - \sqrt{3}.$$
\end{exmp}

 \cref{Summation of Positively Hom} leads us to consider if there exist additional assumptions under which \begin{equation}\label{natural question}
\min_{i \in \{1, \dots, m\}} \{ \alpha(h_i)\} = \min\{ \alpha(h_1), \alpha(h_m)\}\end{equation} in the setting of \cref{Summation of Positively Hom} and thus the symmetry coefficient of $\alpha(h_1 + \dots + h_m)$ is completely known. This is the topic of investigation in the following two sections, where it turns out that \eqref{natural question} holds for sums of functions of the form $\lambda_i\norm{\cdot}_2^{p_i} : \R^n \to \R$, where $\lambda_i > 0$ and $p_i > 1$.

\section[]{The Symmetry Coefficient of $|\cdot|^p$}\label{sym 1d}
In this section we will consider the symmetry coefficient of the one-dimensional Legendre and positively homogeneous function $|\cdot|^p : \R \to \R$ for some $p > 1$. Central to this study is the function $f_p : \R \to \R$ defined by 
    \begin{equation} \label{definition of f_p}
    f_p(u) \coloneqq \begin{cases}
        \frac{|u|^p - pu + (p - 1)}{(p-1)|u|^p - p \sgn(u)|u|^{p-1} + 1} &\text{ if } u \neq 1, \\
        1 &\text{ if } u = 1, \\
    \end{cases}\end{equation} for all $u \in \R$. Moreover, when $p > 2$ we will consider the function $g_p : [0, 1] \to \R$ defined by
    \begin{equation}\label{defintion of g_p}
        g_p(u) \coloneqq u^{2(p-1)} - (p-1)^2 u^p - 2p(p-2)u^{p-1} - (p-1)^2 u^{p-2} + 1
    \end{equation} for all $u \in [0, 1]$. In fact, we will show in \cref{prp::first proposition of p norms} that computing $\alpha(|\cdot|^p)$ is equivalent to computing the extremum points of $f_p$. Furthermore, when $p > 2$ this is equivalent to computing the roots of the function $g_p$ over its domain $[0, 1]$. Even though $f_p$ is generally non-convex and non-concave, we will in \cref{technical lemma} show that if $p > 2$, then $f_p$ is quasiconcave on $[-1, 1]$ and is uniquely maximized over $[-1, 1]$ by some point in $(-1, 0)$. This lets us recast the non-convex and non-concave problem of naively computing $\alpha(|\cdot|^p)$ via \eqref{eq::definition_of_symmetry_coefficient} to a one-dimensional quasiconcave problem over a compact interval, as implemented in \cref{alg:compute symmetry for 1d norm}.

    Not only will the study of $f_p$ lead to an efficient way of computing $\alpha(|\cdot|^p)$, it will also let us show that the function $\psi : (1, \infty) \to (0, 1]$ defined by 
    $
    \psi(p) = \alpha(|\cdot|^p)
    $, for all $p \in (1, \infty)$, is strictly increasing on the interval $(1, 2)$ and strictly decreasing on the interval $(2, \infty)$. This is shown in \cref{Monotonically Decreasing Symmetry Coefficient}. When combining this with the bounds found in \cref{Summation of Positively Hom}, we arrive at a complete characterization of what $\alpha\left( \sum_{i = 1}^m |\cdot|^{p_i} \right)$ equals, which will be shown in \cref{prp::answering-proposition to question}.

\begin{prp} \label{prp::first proposition of p norms}
    Let $p > 1$ and let $f_p : \R \to \R$ as in \eqref{definition of f_p}. Then, the denominator in \eqref{definition of f_p} is positive, i.e., $$(p-1)|u|^p - p \sgn(u)|u|^{p-1} + 1 > 0$$ holds for all $u \in \R \setminus \{1\}$. Furthermore, $f_p$ is continuous, differentiable on $(-1, 1)$, positive, and satisfies $f_p(u^{-1}) = f_p(u)^{-1}$ for all $u \neq 0$. Moreover, the symmetry coefficient of $|\cdot|^p$ is given by \begin{equation}\label{eq::compact domain}\alpha(|\cdot|^p) = \min \biggl\{\min\bigl\{f_p(u), f_p(u)^{-1}\bigr\} \mid u \in [-1, 1]\biggr\}.\end{equation}
\end{prp}
\begin{proof}
    First, note that $|\cdot|^p : \R \to \R$ is proper, closed, differentiable on $\R$ and strictly convex, and so it is a Legendre function. Let $x, y \in \R^n$ such that $x \neq y$, then we have that 
    \begin{align} \label{eq::f(u) calculation}
    \frac{D_{|\cdot|^p}(x, y)}{D_{|\cdot|^p}(y, x)} &= \frac{|x|^p - |y|^p - p \sgn(y) |y|^{p-1}(x-y)}{|y|^p - |x|^p - p \sgn(x) |x|^{p-1}(y-x)}\nonumber \\
    &= \frac{|x|^p + (p-1)|y|^p - p\sgn(y)|y|^{p-1}x}{(p-1)|x|^p + |y|^p - p\sgn(x)|x|^{p-1}y} \nonumber\\
    &= \begin{cases}
        \frac{|u|^p - pu + (p - 1)}{(p-1)|u|^p - p \sgn(u)|u|^{p-1} + 1} &\text{ if } y \neq 0 \text{ and } u \coloneqq x/y, \\
        (p-1)^{-1} &\text{ if } y = 0, \\
    \end{cases}
    \end{align} where the last equality for $y \neq 0$ is obtained by dividing the numerator and denominator by $|y|^p$ and simplifying.
    By the strict convexity of $|\cdot|^p$, we have that $D_{|\cdot|^p}(y, x) \in \R_{++}$ holds for all $x, y \in \R$ such that $x \neq y$, and since the first of the two expressions in \eqref{eq::f(u) calculation} equals $f_p(u)$ for $u \neq 1$, we have that the denominator in \eqref{definition of f_p} is positive,
    
    Two uses of l'Hospital's rule give that 
    \begin{align*}
    \lim_{u \to 1} f_p(u) &= \lim_{u \to 1} \frac{u^p - pu + (p - 1)}{(p-1)u^p - p u^{p-1} + 1} \\
    &= \lim_{u \to 1} \frac{pu^{p-1} - p}{(p-1)pu^{p-1} - (p-1)p u^{p-2}} \\
    &= \lim_{u \to 1} \frac{(p-1)pu^{p-2}}{(p-1)^2pu^{p-2} - (p-2)(p-1)p u^{p-2}} = 1
    \end{align*} 
    and thus $f_p$ is continuous. 

    By the definition of $f_p$ in \eqref{definition of f_p} and the fact that the denominator in \eqref{definition of f_p} is positive, the function $f_p$ is differentiable on the set $(-1, 1)$. By strict convexity of $|\cdot|^p$, we have that $D_{|\cdot|^p}(x, y) > 0$ for all $x, y \in \R$ such that $x \neq y$, and so $f_p$ is positive. If $u \neq 0$, then by \eqref{eq::f(u) calculation} 
    \begin{align} \label{eq::f(u) inverse property}
        f_p(u) = \frac{D_{|\cdot|^p}(|u|, \sgn(u))}{D_{|\cdot|^p}(\sgn(u), |u|)} &= \left(\frac{D_{|\cdot|^p}(\sgn(u), |u|)}{D_{|\cdot|^p}(|u|, \sgn(u))}\right)^{-1}\nonumber \\
        &= \left(\frac{D_{|\cdot|^p}(|u|u^{-1}, |u|)}{D_{|\cdot|^p}(|u|, |u|u^{-1})}\right)^{-1}\nonumber \\
    &= \left(\frac{|u|^p D_{|\cdot|^p}(u^{-1}, 1)}{|u|^p D_{|\cdot|^p}(1, u^{-1})}\right)^{-1} = f_p(u^{-1})^{-1}
    \end{align} 
    where the fourth equality follows the fact that $D_{|\cdot|^p}$ is positively homogeneous of degree $p$, as shown in \cref{cor::domains}. Therefore, \begin{align*}
        \alpha(|\cdot|^p) &= \inf\left\{\left. \frac{D_{|\cdot|^p}(x, y)}{D_{|\cdot|^p}(y, x)} \right| x, y \in \R\text{ and } x\neq y \right \} \\
        &= \min \biggl\{\inf\bigl\{f_p(u) \mid u \in \R \setminus \{1\}\bigr\}, (p-1)^{-1}  \biggr\} \\
        &= \inf \biggl\{\min\bigl\{f_p(u), f_p(u)^{-1}\bigr\} \mid u \in \R\biggr\} \\
        &= \min \biggl\{\min\bigl\{f_p(u), f_p(u)^{-1}\bigr\} \mid u \in [-1, 1]\biggr\},
    \end{align*} where we in the second equality used \eqref{eq::f(u) calculation}, and in the third equality used that $f_p>0$ and that $f_p(0) = p-1$, and in the last equality used \eqref{eq::f(u) inverse property} combined with the fact that $f_p$ is continuous.
\end{proof}
\begin{cor} \label{cor::simple consequence of ||p}
    $\lim_{q \downarrow 1} \alpha(|\cdot|^q) = \lim_{p \to \infty}\alpha(|\cdot|^p) = 0$.
\end{cor}
\begin{proof}
    Let $p > 2$. By \cref{prp::first proposition of p norms}, $\alpha(|\cdot|^p) \leq f_p(0)^{-1} = (p-1)^{-1} \to 0$ as $p \to \infty$. The conjugate relation $(|\cdot|^p/p)^* = |\cdot|^q/q$, which holds for $q = p/(p-1) \in (1, 2)$, combined with \cref{basic proposition} implies that $$\alpha(|\cdot|^p) = \alpha(|\cdot|^p/p) = \alpha((|\cdot|^p/p)^*) =  \alpha(|\cdot|^q/q) = \alpha(|\cdot|^q),$$ which implies that $\lim_{q \downarrow 1} \alpha(|\cdot|^q) = 0$. 
\end{proof}

\begin{prp}\label{technical lemma}
    Let $p > 2$, $f_p : \R \to \R$ be as in \eqref{definition of f_p}, and $g_p :[0, 1] \to \R$ as in \eqref{defintion of g_p}. Then 
    \begin{enumerate}[label=(\roman*)]
        \item $f_p$ is quasiconcave on $[-1, 1]$ and is uniquely maximized over $[-1, 1]$ by a point $u_0$ satisfying $$u_0 \coloneqq \argmax_{u \in [-1, 1]} f_p(u) \in (-1, 0).$$ 
        Furthermore, $u_0$ is the unique stationary point of $f_p$ over $(-1, 0)$ and $-u_0$ is the unique root of $g_p$ over $[0, 1]$.
        \label{technical lemma i}
        \item if $q > p$ then $f_q(u) > f_p(u)$ for all $u \in (-1, 0)$.\label{technical lemma ii}
    \end{enumerate}
\end{prp}
\begin{proof} 
\mbox{} \smallskip
\begin{enumerate}[labelindent=0pt,wide=0pt]
\item[\labelcref{technical lemma i}]
    Fix $p > 2$. By \cref{prp::first proposition of p norms}, $f_p$ is continuous on $[-1, 1]$ and differentiable on $(-1, 1)$. Let $E \coloneq (-1, 1) \setminus \{0\}$. 
    
    Consider the functions $g_{1,p}, g_{2,p}: [-1, 1] \to \R$ and $g_{3,p}: [-1, 1] \setminus \{0\} \to \R$ defined by \begin{align}
        g_{1,p}(u) &\coloneqq |u|^{2(p-1)} - (p-1)^2|u|^p + 2p(p-2) \sgn(u)|u|^{p-1}-(p-1)^2|u|^{p-2} + 1, \label{eq::g1 def}\\
        g_{2,p}(u) &\coloneqq 2|u|^p - p(p-1)u^2 + 2p(p-2)u - (p-1)(p-2), \nonumber\\
        g_{3,p}(u) &\coloneqq |u|^{p-1} - (p-1)|u| + (p-2)\sgn(u),\nonumber
    \end{align} for all $u \in [-1, 1]$, $u \in [-1, 1]$ and $u \in [-1,1] \setminus\{0\}$, respectively. It is immediate that $g_{i,p}$ is differentiable on $E$ for all $i \in \{1, 2, 3\}$. Furthermore, the following equations \begin{align} 
        f_p'(u)  &= -pg_{1,p}(u)\bigg/\biggl((p-1)|u|^{p} - p\sgn(u) |u|^{p-1} + 1 \biggr)^2, \label{lemma 1 eq 1} \\
        g_{1,p}'(u) &= (p-1)\sgn(u)|u|^{p-3}g_{2,p}(u), \label{lemma 1 eq 2} \\
        g_{2,p}'(u) &= 2p\sgn(u)g_{3,p}(u), \label{lemma 1 eq 3} \\
        g_{3,p}'(u) &= (p-1)\sgn(u)(|u|^{p-2} - 1), \label{lemma 1 eq 4}
    \end{align} can be verified to hold for all $u \in E$. Note that, in particular, $f_p'$ is continuous on $E$. This fact, combined with $f_p(-1) = f_p(1) = 1 < p-1 = f_p(0)$ and that $f_p$ is continuous on $[-1, 1]$, gives that the statement: \begin{equation}\label{Claim}
        f_p' \text{ has exactly one zero on } E
    \end{equation} implies that $f_p$ is quasiconcave on $[-1, 1]$. We proceed by showing that \eqref{Claim} is true.

    One can verify that
    \begin{align}
        g_{1,p}(-1) &= 4p(2-p) < 0, \quad g_{1,p}(0) = 1, \quad g_{1,p}(1) = 0 \label{lemma 1 eq 5}\\
        g_{2,p}(-1) &= 4p(2-p) < 0,\quad g_{2,p}(0) = -(p-1)(p-2) < 0, \quad g_{2,p}(1) = 0 \label{lemma 1 eq 6}\\
        g_{3,p}(-1) &= 2(2-p) < 0,\quad \lim_{x \uparrow 0}g_{3,p}(x) = 2-p < 0 < p - 2 = \lim_{x \downarrow 0}g_{3,p}(x), \quad g_{3,p}(1) = 0 \label{lemma 1 eq 7}
        \end{align}
        and that
    \begin{equation}
        g_{3,p}'(u) \neq 0 \text{ for all } u \in E. \label{lemma 1 eq 8}
    \end{equation} Using the mean value theorem on $g_{3,p}$ with \eqref{lemma 1 eq 7} and \eqref{lemma 1 eq 8} we get that $g_{3,p}(u) \neq 0$ for all $u \in E$. Therefore, by \eqref{lemma 1 eq 3}, $g_{2,p}'(u) \neq 0$ for all $u \in E$. Using this and the mean value theorem on $g_{2,p}$ with \eqref{lemma 1 eq 6} we get that $g_{2,p}(u) < 0$ for all $u \in E$. Therefore, by \eqref{lemma 1 eq 2}, $g_{1,p}'(u) \neq 0$ for all $u \in E$.  Using this and the mean value theorem on $g_{1,p}$ with \eqref{lemma 1 eq 5} we get that there exists a unique $u_0 \in (-1, 0)$ such that $g_{1,p}(u_0) = 0$ and that $g_{1,p}(u) > 0$ holds for all $u \in (0, 1)$. Therefore, by \eqref{lemma 1 eq 1}, \eqref{Claim} holds and $u_0 = \argmax_{u \in [-1, 1]} f_p(u) \in (-1, 0)$. Furthermore, $f_p$ is quasiconcave on $[-1, 1]$. Since $g_p(u) = g_{1, p}(-u)$ holds for all $u \in [0, 1]$, we have that $-u_0 \in (0, 1)$ is the unique root of $g_p$ over $[0, 1]$.

\item[\labelcref{technical lemma ii}]
    Fix $q > p > 2$. We will show that 
    \begin{equation}\label{Claim 2}
    f_q(-u) > f_p(-u) \text{ for all } u \in (0, 1).
    \end{equation}  
    For all $u \in (0, 1)$ we have that \begin{equation}\label{eq:fq-u}f_q(-u) = \frac{u^q + qu + (q-1)}{(q-1)u^q + qu^{q-1} + 1}.\end{equation} Consider the functions $g_{1,p,q}, g_{2,p,q} : (0, 1) \to \R$ defined by 
    \begin{align}
        g_{1,p,q}(u) &\coloneqq (u^q + qu + (q - 1))((p-1)u^p + pu^{p-1} + 1) \nonumber\\
        &\ - (u^p + pu + (p - 1))((q-1)u^q + qu^{q-1} + 1), \nonumber \\
        g_{2,p,q}(u) &\coloneqq p(q-1)(1 - u^{q-p+1}) + q(p-1)(u - u^{q-p}), \label{lemma 1 eq 11}
    \end{align} for all $u \in (0, 1)$. It can be verified that 
    \begin{align}
        f_q(-u) - f_p(-u)  &= g_{1,p,q}(u)\bigg/\left(\left((q-1)u^q + qu^{q-1} + 1\right)\left((p-1)u^p + pu^{p-1} + 1\right)\right) \label{lemma 1 eq 9} \\
        g_{2,p,q}(u)  &= \frac{1}{u^{p-1}}\left(\frac{1}{u+1} g_{1,p,q}(u) + (q-p)(u^{p+q-1} - 1)\right) \label{lemma 1 eq 10}
    \end{align} holds for all $u \in (0, 1)$. Since $(q-p)(u^{p+q - 1}-1) < 0$ holds for all $u \in (0, 1)$ we have that the statement 
    \begin{equation}\label{Claim 3}
        g_{2,p,q}(u) > 0 \text{ for all } u \in (0, 1)
    \end{equation} implies, by \eqref{lemma 1 eq 10}, that $g_{1,p,q}(u) > 0$ for all $u \in (0, 1)$, which in turn implies, by \eqref{lemma 1 eq 9}, that \eqref{Claim 2} holds. We will now focus on showing \eqref{Claim 3}.
    
     If $q-p \in [1, \infty)$, then \eqref{Claim 3} holds directly by \eqref{lemma 1 eq 11}, since then $u-u^{q-p} \geq 0$ holds for all $u \in (0, 1)$.

    The remaining case is when $q-p \in (0, 1)$. Introduce $s_1 \coloneqq p(q-1)(q-p+1)$, $s_2 \coloneqq q(p-1)(p - q+1)$. We compute the derivatives 
    \begin{align}
        g_{2,p,q}'(u) &= -s_1u^{q-p} + q(p-1) - q(p-1)(q-p)u^{q-p-1}, \label{lemma 1 eq 13} \\
        g_{2,p,q}''(u) &= (q-p)u^{q-p-2}(s_2 - s_1u). \label{lemma 1 eq 12}
    \end{align} 
    and verify that \begin{align*}
        \lim_{u \downarrow 0} g_{2,p,q}(u) = p(q-1) > 0&, \quad  \lim_{u \uparrow 1} g_{2,p,q}(u) = 0, \\ 
        \lim_{u \downarrow 0} g_{2,p,q}'(u) = -\infty&, \quad \lim_{u \uparrow 1} g_{2,p,q}'(u) = (p-q)\underbrace{(p(2q-1)-q+1)}_{ > p-q+1 > 0} < 0,
    \end{align*} from which we see that the statement \begin{equation} \label{Claim 4}
         g_{2,p,q}'(u) < 0 \text{ for all } u \in (0, 1)
    \end{equation} implies together with the mean value theorem applied to $g_{2,p,q}$ that \eqref{Claim 3} holds. 
    
    Suppose that \eqref{Claim 4} does not hold, i.e., there exists some $u \in (0, 1)$ such that $g_{2,p,q}'(u) \geq 0$. Then, by the mean value theorem applied to $g_{2,p,q}'$, there exists some $\bar{u} \in (0, 1)$ such that $g_{2,p,q}''(\bar{u}) = 0$ and $g_{2,p,q}'(\bar{u}) \geq 0$. By \eqref{lemma 1 eq 12}, it is necessary that $\bar{u} = s_2/s_1 \in (0, 1)$. By \eqref{lemma 1 eq 13}, we have that \begin{align*}
        g_{2,p,q}'(\bar{u}) &= -s_1\bar{u}^{q-p} + q(p-1) - q(p-1)(q-p)\bar{u}^{q-p-1}\\
        &= -s_2\bar{u}^{q-p-1} + q(p-1) - q(p-1)(q-p)\bar{u}^{q-p-1}\\
        &= q(p-1)\left( -(p-q+1)\bar{u}^{q-p-1} + 1 - (q-p)\bar{u}^{q-p-1}\right) = q(p-1)(1 - \bar{u}^{q-p-1})
    \end{align*} and since $q-p \in (0, 1)$ we get that $g_{2,p,q}'(\bar{u}) < 0$, which is a contradiction. Therefore, \eqref{Claim 4} holds.
\end{enumerate}
\end{proof}

\begin{cor}\label{corollary to lemmas}
 Let $p > 2$ and $f_p : \R \to \R$ be as in \eqref{definition of f_p}. Then, $\alpha(|\cdot|^p) > 0$ and \begin{equation*}\alpha(|\cdot|^p) = \biggl(\max  \left\{f_p(u) \mid u \in (-1, 0)\right\}\biggr)^{-1}.\end{equation*} 
\end{cor}
\begin{proof}
    By \cref{prp::first proposition of p norms} and \itemcref{technical lemma}{technical lemma i}, the function $f_p$ is continuous, positive, and quasiconcave on $[-1, 1]$ and satisfies $f_p(-1) = f_p(1) = 1$. Therefore, we get that $\alpha(|\cdot|^p) > 0$ and that $f_p(u)^{-1} \leq f_p(u)$ holds for all $u \in [-1, 1]$. Combining this with \eqref{eq::compact domain} we get that \begin{align*}
        \alpha(|\cdot|^p) &= \min \biggl\{f_p(u)^{-1} \mid u \in [-1, 1]\biggr\} \\
        &= \biggl(\max\ \{f_p(u) \mid u \in [-1, 1]\}\biggr)^{-1} \\
        &= \biggl(\max\ \{f_p(u) \mid u \in (-1, 0)\}\biggr)^{-1},
    \end{align*} where the last equality follows from \itemcref{technical lemma}{technical lemma i}.
\end{proof}

\begin{rem}
    Combining \cref{prp::first proposition of p norms}, \cref{technical lemma} and \cref{corollary to lemmas} yields \cref{alg:compute symmetry for 1d norm} which computes $\alpha(|\cdot|^p)$ for any $p \in (1, \infty)$. \cref{line::conjugacy} in \cref{alg:compute symmetry for 1d norm} is valid by the conjugacy argument described in the proof of \cref{cor::simple consequence of ||p}. \itemcref{technical lemma}{technical lemma i} gives that the bisection method in \cref{line::bisection} is well-defined. The method $\texttt{BisectionMethod}$ in \cref{line::bisection} is intentionally formulated in a general way, as any standard implementation of the bisection method suffices to guarantee the correctness of \cref{alg:compute symmetry for 1d norm}.
\end{rem}

\begin{algorithm}[H]
	\caption{Computing $\alpha(|\cdot|^p$)}
	\begin{algorithmic}[1]
	    \Statex \hspace{-5.7mm} \textbf{Input:} $p \in (1, \infty)$
            \State \textbf{Let:} $f_p$ as in \eqref{definition of f_p} and $g_p$ as in \eqref{defintion of g_p}
            \If {$p = 2$}
                \State \textbf{Output:} $1$
            \EndIf
            \If {$p < 2$}
                \State $p \leftarrow p/(p-1)$ \label{line::conjugacy}
            \EndIf
            \State \textbf{Compute:} $u_0 \leftarrow \texttt{BisectionMethod}(g_p, \ \text{left\_endpoint}=0, \ \text{right\_endpoint}=1)$ \label{line::bisection}
            \State \textbf{Output:} $f_p(-u_0)^{-1}$
	\end{algorithmic}
\label{alg:compute symmetry for 1d norm}
\end{algorithm}

\begin{cor}\label{Monotonically Decreasing Symmetry Coefficient}
    Let $q_1 > p_1 \geq 2 \geq p_2 > q_2.$ Then, $\alpha(|\cdot|^{p_1}) > \alpha(|\cdot|^{q_1})$ and $\alpha(|\cdot|^{p_2}) > \alpha(|\cdot|^{q_2})$.
\end{cor}
\begin{proof}
    By the conjugacy relation in \cref{basic proposition}, $\alpha(|\cdot|^{p_2}) = \alpha(|\cdot|^{p_2/(p_2-1)})$, and since $$q_2/(q_2-1) > p_2/(p_2-1) \geq 2$$ we only need to prove the statement regarding $q_1$ and $p_1$.
    
    If $p_1 = 2$, then the statement follows from \cref{prop::perfect_symmetry}. Otherwise, by \cref{corollary to lemmas} and \itemcref{technical lemma}{technical lemma ii} we have that \begin{align*}
        \alpha(|\cdot|^{q_1})^{-1} &= \max  \bigl\{f_q(u) \mid u \in (-1, 0)\bigr\} > \max  \bigl\{f_p(u) \mid u \in (-1, 0)\bigr\} = \alpha(|\cdot|^{p_1})^{-1}
    \end{align*} which implies the conclusion.
\end{proof}

We now return to the question posed in \eqref{natural question} and answer it in this setting.

\begin{thm} \label{prp::answering-proposition to question}
    Let $1 < p_1 \leq p_2 \leq \dots \leq p_m$ and let $\lambda_i \in \R_{++}$ for each $i \in \{1, \dots, m\}$. Then, $$\alpha\left(\sum_{i = 1}^m\lambda_i|\cdot|^{p_i}\right) = \min\{\alpha(|\cdot|^{p_1}), \alpha(|\cdot|^{p_m})\}.$$
\end{thm}
\begin{proof}
    We have that \begin{align*}
        \alpha\left(\sum_{i = 1}^m\lambda_i|\cdot|^{p_i}\right) &\in \left[\min_{i \in\{1, \dots, m\}} \{ \alpha(\lambda_i |\cdot|^{p_i})\}, \ \min\{ \alpha(\lambda_1 |\cdot|^{p_1}), \alpha(\lambda_m |\cdot|^{p_m})\}\right] \\
         &= \left[\min_{i \in\{1, \dots, m\}} \{ \alpha(|\cdot|^{p_i})\}, \ \min\{ \alpha(|\cdot|^{p_1}),\alpha(|\cdot|^{p_m})\}\right] \\
         &=  \biggl\{\min\{ \alpha(|\cdot|^{p_1}),\alpha(|\cdot|^{p_m})\}\biggr\},
    \end{align*} where the inclusion follows from \cref{Summation of Positively Hom} and since $\lambda_i |\cdot|^{p_i}$ is positively homogeneous of degree $p_i$ for all $i \in \{1, \dots, m\}$, the first equality follows from \cref{basic proposition}, and the second equality follows from \cref{Monotonically Decreasing Symmetry Coefficient}.
\end{proof}

We end this section with the asymptotic result relating to \labelcref{item::C_3}.

\begin{thm} \label{thm:assymptotic theorem}
\mbox{} \smallskip
\begin{enumerate}[label=(\roman*)]
    \item  $\alpha(|\cdot|^p) > (2p)^{-1}$ holds for all $p \in (2, \infty)$. \label{itm:assymptotic lower bound}
    \item  For any $c \in (0, 2)$ there exists a $P \in (2, \infty)$ such that $p \in (P, \infty)$ implies that $$\alpha(|\cdot|^p) < (cp)^{-1}.$$ \label{itm:not assymptotic lower bound}
    \item  $\alpha(|\cdot|^p) \sim(2p)^{-1}$ as $p \to \infty$. \label{itm:assymptotic equivalence}
\end{enumerate}
\end{thm}
\begin{proof}
Let $p \in (2, \infty)$ and $c \in (1, \infty)$. Let $g_{1, p, c}, g_{2, p, c} : (0, 1) \to \R$ be defined by 
\begin{align*}
    g_{1, p, c}(u) &\coloneqq \left(cp(p-1) - 1\right)u^p + cp^2 u^{p-1} - pu + p(c-1) + 1 \\
    g_{2, p, c}(u) &\coloneqq \left(cp(p-1) - 1\right)u^{p-1} + cp(p-1)u^{p-2} - 1 
\end{align*} for all $u \in (0, 1)$. These functions are continuously differentiable on $(0, 1)$ and \begin{align}
    f_p(-u)^{-1} - (cp)^{-1} &= g_{1, p, c}(u)/\left( cp(u^p + pu + p -1)\right) \label{eq:assymptotic 1} \\
    g_{1, p, c}'(u) &= p g_{2, p, c}(u) \label{eq:assymptotic 2}
\end{align} holds for all $u \in (0, 1)$, where we have utilized \eqref{eq:fq-u} in \eqref{eq:assymptotic 1}. Using \cref{corollary to lemmas,eq:assymptotic 1} we see that \begin{equation} \label{eq:assymptotic eq1}
\alpha(|\cdot|^p) - (cp)^{-1} > 0 \iff g_{1, p, c}(u) > 0 \text{ for all } u \in (0, 1)
\end{equation} and \begin{equation} \label{eq:assymptotic eq2}
\alpha(|\cdot|^p) - (cp)^{-1} < 0 \iff g_{1, p, c}(u) < 0 \text{ for some } u \in (0, 1).
\end{equation} Since $\lim_{u \downarrow 0} g_{1, p, c}(u) = p(c-1)+1 > 1$ and $\lim_{u \uparrow 1} g_{1, p, c}(u) = 2p(cp-1) > 4$, the only chance for $g_{1, p, c}$ being negative is at stationary points, i.e., points where $g_{2, p, c}$ is zero, by \eqref{eq:assymptotic 2}. Since $\lim_{u \downarrow 0} g_{2, p, c}(u) = -1$ and $\lim_{u \uparrow 1} g_{2, p, c}(u) = 2(cp(p-1) - 1) > 2$ there exists a point $u_{p, c} \in (0, 1)$ such that $g_{2, p, c}(u_{p, c} ) = 0$. Since $g_{2, p, c}$ is a strictly increasing function on $(0, 1)$, the zero $u_{p, c}$ of $g_{2, p, c}$ must be unique, and we will, for each $c \in (1, \infty)$ and $p \in (2, \infty)$, from now on let $u_{p, c}$ denote the unique zero of $g_{2, p, c}$. By the uniqueness of $u_{p,c}$ we get the updated equivalences from \eqref{eq:assymptotic eq1} and \eqref{eq:assymptotic eq2} as
\begin{equation} \label{eq:assymptotic eq1 updated}
\alpha(|\cdot|^p) - (cp)^{-1} > 0 \iff g_{1, p, c}(u_{p,c}) > 0
\end{equation} and \begin{equation} \label{eq:assymptotic eq2 updated}
\alpha(|\cdot|^p) - (cp)^{-1} < 0 \iff g_{1, p, c}(u_{p,c}) < 0,
\end{equation} respectively. By definition of $u_{p,c}$ we get that 
\begin{align*}u_{p,c}g_{2, p, c}(u_{p, c}) &= \left(cp(p-1) - 1\right)u_{p, c}^{p} + cp(p-1)u_{p,c}^{p-1} - u_{p,c} = 0\end{align*} and so
\begin{align}
    g_{1, p, c}(u_{p,c}) &= \left(cp(p-1) - 1\right)u_{p,c}^p + cp^2 u_{p,c}^{p-1} - pu_{p,c} + p(c-1) + 1 \nonumber \\
    &= -cp(p-1)u_{p,c}^{p-1} + u_{p,c} + cp^2 u_{p,c}^{p-1} - pu_{p,c} + p(c-1) + 1 \nonumber \\
    &= cpu_{p,c}^{p-1} + (1-p)u_{p,c} + p(c-1) + 1 \nonumber \\
    &= p\left(c u_{p,c}^{p-1} + c - (u_{p,c} + 1)\right) + u_{p,c} + 1. \label{eq:assymptotic eq3}
\end{align} 

\begin{enumerate}[labelindent=0pt,wide=0pt]
\item[\labelcref{itm:assymptotic lower bound}] In the case $c = 2$, then $$p\left(c u_{p,c}^{p-1} + \underbrace{c - (u_{p,c} + 1}_{> 0})\right) + u_{p,c} + 1$$ and so $g_{1, p, c}(u_{p,c}) > 0$ which with \eqref{eq:assymptotic eq1 updated} in mind implies \labelcref{itm:assymptotic lower bound}.
\item[\labelcref{itm:not assymptotic lower bound}] If the statement holds for some $c \in (1, 2)$, then it also must hold for any $c' \in (0, c)$, since $(cp)^{-1} < (c'p)^{-1}$ holds for any $p > 0$. Therefore, let us focus on that case and fix $c \in (1, 2)$. 

First, we will establish that $\lim_{p \to \infty} u_{p, c}^{p-1} = 0$. This follows from the defining property $g_{2, p, c}( u_{p, c}) = 0$, since we get that $$0 < u_{p, c}^{p-1} = \frac{1 - cp(p-1) u_{p, c}^{p-2}}{cp(p-1)-1} < \frac{1}{cp(p-1)-1} \to 0 \text{ as } p \to \infty.$$ Next, we will establish that $\lim_{p \to \infty} u_{p, c} = 1$. We do this by showing that for any $u \in (0, 1)$ then $\lim_{p \to \infty} g_{2, p, c}(u) = -1$. Since $g_{2, p, c}$ is an increasing function, this will, by the mean value theorem applied to $g_{2, p, c}$, imply that $\lim_{p \to \infty} u_{p, c} = 1$. Let $u \in (0, 1)$, then $$-1 < g_{2, p, c}(u) < cp^2u^{p-1} + cp^2u^{p-2} - 1 \to -1 \text{ as } p \to \infty.$$

Now since $\lim_{p \to \infty}c - (u_{p,c} + 1) = c-2 < 0$, there exists a $P_1 \in (2, \infty)$ be such that $p \in (P_1, \infty)$ implies that $c - (u_{p,c} + 1) < (c-2)/2$. Since $\lim_{p \to \infty}u_{p,c}^{p-1} = 0$, there exists a $P_2 \in (2, \infty)$ such that $p \in (P_2, \infty)$ implies that $u_{p,c}^{p-1} < (2-c)/(4c)$. Finally, let $P_3 \in (8/(2-c), \infty)$ and $P \coloneqq \max\{P_1, P_2, P_3\}$. Then if $p \in (P, \infty)$ we get that \begin{align*}
    g_{1, p, c}(u_{p,c}) &= p\left(c u_{p,c}^{p-1} + c - (u_{p,c} + 1)\right) + u_{p,c} + 1 \\
    &<  p\left(c u_{p,c}^{p-1} + c - (u_{p,c} + 1)\right) + 2 \\
    &< p\left(c u_{p,c}^{p-1} + \tfrac{c-2}{2}\right) + 2 \\
    &< p \tfrac{c-2}{4} + 2  \\
    &< 0,
\end{align*} where the equality is \eqref{eq:assymptotic eq3}, the first inequality holds since $u_{p,c}<1$, and the second, third, and fourth inequalities follow from $p > P_1$, $p>P_2$, and $p>P_3$, respectively. Therefore, from \eqref{eq:assymptotic eq2 updated}, \labelcref{itm:not assymptotic lower bound} follows.
\item[\labelcref{itm:assymptotic equivalence}] Let $\phi : (2, \infty) \to \R$ be defined by $\phi(p) = 2p\alpha(|\cdot|^p)$ for all $p \in (2, \infty)$. We aim to show that $\lim_{p \to \infty}\phi(p) = 1$. By \labelcref{itm:assymptotic lower bound}, $\phi > 1$ and so $\liminf_{p \to \infty} \phi(p) \geq 1$. Now suppose that $\limsup_{p \to \infty} \phi(p) > 1$, then there exists some $c \in (0, 1)$ such that $$\limsup_{p \to \infty} c\phi(p) = \limsup_{p \to \infty} 2cp\alpha(|\cdot|^p) > 1,$$ i.e., for any $P \in (2, \infty)$ there exists some $p \in (P, \infty)$ such that $\alpha(|\cdot|^p) > (2cp)^{-1}$. But $2c \in (0, 2)$, and so this contradicts \labelcref{itm:not assymptotic lower bound}. Therefore, $\limsup_{p \to \infty} \phi(p) \leq 1 \leq \liminf_{p \to \infty} \phi(p)$ which implies that $\lim_{p \to \infty} \phi(p) = 1$, i.e, $\alpha(|\cdot|^p) \sim (2p)^{-1}$.
\end{enumerate}
\end{proof}

\section[]{The symmetry coefficient of $\norm{\cdot}_p^p$ and $\norm{\cdot}_2^p$} \label{sym nd}

This section will be dealing with the symmetry coefficients of two types of functions, $\norm{\cdot}_p^p : \R^n \to \R$ and $\norm{\cdot}_2^p : \R^n \to \R$, of general underlying dimension $n$. As \cref{p norm simple}, the symmetry coefficient of $\norm{\cdot}_p^p$ is equal to its 1-dimensional counterpart, as treated in \cref{Symmetry when Separable}. 

\begin{cor}\label{p norm simple}
    Let $p > 1$ and $n$ be a positive integer. Consider the function $\norm{\cdot}_p^p : \R^n \to \R$. Then, $\alpha(\norm{\cdot}_p^p) = \alpha(|\cdot|^p).$
\end{cor}
\begin{proof}
    Since $\norm{x}_p^p = \sum_{i = 1}^n |x_i|^p$ holds for all $x = (x_1, \dots, x_n) \in \R^n$, the function $\norm{\cdot}_p^p$ is separable and by \cref{Symmetry when Separable}, $$\alpha(\norm{\cdot}_p^p) = \min_{i \in \{1, \dots, n\}} \{\alpha(|\cdot|^p)\} = \alpha(|\cdot|^p).$$
\end{proof}

As \cref{p norm hard}, the symmetry coefficient of $\norm{\cdot}_p^2$ is fully captured by the one-dimensional case covered in \cref{corollary to lemmas}. 

\begin{thm}\label{p norm hard}
    Let $p > 1$ and $n \geq 2$ be a positive integer. Consider the function $\norm{\cdot}_2^p : \R^n \to \R$. Then,  $\alpha(\norm{\cdot}_2^p) = \alpha(|\cdot|^p)$.
\end{thm}
\begin{proof}
    First, note that $\norm{\cdot}_2^p : \R^n \to \R$ is a Legendre function.
    
    If $p = 2$, then the conclusion follows as a special case of \cref{prop::perfect_symmetry}. If $p < 2$, then let $q = p/(p-1) > 2$ so that by \cref{basic proposition} we have that $$\alpha(\norm{\cdot}_2^p) = \alpha\left(\left(\frac{1}{p}\norm{\cdot}_2^p\right)^*\right) = \alpha\left(\frac{1}{q}\norm{\cdot}_2^q\right) = \alpha\left(\norm{\cdot}_2^q\right).$$ Therefore, we only need to consider the case of when $p > 2$. 
    
    Suppose that $p > 2$ and let $x, y \in \R^n$. By the Cauchy-Schwarz inequality, we can write $\langle x, y \rangle = \norm{x}_2 \norm{y}_2 r_{x, y}$ for some  $r_{x, y} \in [-1, 1]$. We get that \begin{align*}
        D_{\norm{\cdot}_2^p}(x, y) &= \norm{x}_2^p - \norm{y}_2^p - \langle p\norm{y}_2^{p-2}y, x-y\rangle  \\
        &= \norm{x}_2^p + (p-1)\norm{y}_2^p - p \norm{x}_2 \norm{y}_2^{p-1} r_{x, y}
    \end{align*} and that
    \begin{equation}\label{alpha H exp1}
    \alpha(\norm{\cdot}_2^p) = \inf \left\{\left. \frac{\norm{x}_2^p + (p-1)\norm{y}_2^p - p\norm{x}_2 \norm{y}_2^{p-1} r_{x, y} }{(p-1) \norm{x}_2^p +\norm{y}_2^p - p\norm{x}_2^{p-1} \norm{y}_2 r_{x, y}} \right| x, y \in \R^n \text{ and } x \neq y \right\}.
    \end{equation} If $y = 0$, then \begin{equation} \label{eq::including y = 0}
         \frac{D_{\norm{\cdot}_2^p}(x, 0)}{ D_{\norm{\cdot}_2^p}(0, x)} = \frac{1}{p-1}
    \end{equation} holds for all $x \in \R^n \setminus \{0\}$. 
    
    Let $K \coloneqq (\R_+ \times [-1, 1]) \setminus (\{1\} \times \{1\})$ and $F_p : K \to \R$ be defined by \begin{equation}\label{eq::def F_p}
    F_p(u, r) \coloneqq \frac{u^p - pru + (p-1)}{(p-1)u^p - pru^{p-1} + 1},
    \end{equation} for all $(u, r) \in K$, and we will shortly show that the right-hand side of \eqref{eq::def F_p} is well defined over $K$, i.e., that the denominator of \eqref{eq::def F_p} is non-zero for all $(u, r) \in K$. If $y \neq 0$, then 
    \begin{equation}
        \label{eq::excluding y = 0}
        \frac{D_{\norm{\cdot}_2^p}(x, y)}{ D_{\norm{\cdot}_2^p}(y, x)} = F_p\left(\frac{\norm{x}_2}{\norm{y}_2}, r_{x, y}\right).
    \end{equation} Note that since $\norm{\cdot}_2^p$ is strictly convex, $D_{\norm{\cdot}_2^p}(y, x) > 0$ and the denominator in \eqref{eq::def F_p} is positive. Since $x = y$ if and only if $\norm{x}_2/\norm{y}_2 = 1$ and $r_{x, y} = 1$, together with $\{1\} \times \{1\} \notin K$, we get that $F_p$ is well defined. Furthermore, $F_p$ is continuous and positive. Note that, conversely, since $n \geq 2$, we can for any $u \in \R_+$ and any $r \in [-1, 1]$ find some $x, y \in \R^n$ such that $y \neq 0$, $u = \norm{x}_2/\norm{y}_2$, and $\langle x, y \rangle = \norm{x}_2 \norm{y}_2 r$.  Combining this statement, with \eqref{eq::including y = 0}, \eqref{eq::def F_p}, and \eqref{eq::excluding y = 0}, we can now write \eqref{alpha H exp1} as
    \begin{equation} \label{alpha H exp2}
        \alpha(\norm{\cdot}_2^p) = \min\left\{\frac{1}{p-1}, \inf_{(u, r) \in K}\{ F_p(u, r)\}\right\}.
    \end{equation} For any $(u, r) \in K$ with $u > 0$, it follows from \eqref{eq::def F_p} that 
    \begin{equation}\label{F_p fund}F_p(u^{-1}, r) = F_p(u, r)^{-1}.\end{equation} If we combine \eqref{F_p fund} with the fact that $F_p$ is continuous and that $\lim_{u \downarrow 0} F_p(u, r) = p-1$ for all $r \in [-1, 1]$, we can simplify \eqref{alpha H exp2} to get
    \begin{equation} \label{alpha H exp3}
        \alpha(\norm{\cdot}_2^p) = \inf_{(u, r) \in K}\{F_p(u, r)\}.
    \end{equation} Partition $K = K_0 \cup K_1$ where $K_0 = ([0, 1] \times [-1, 1]) \setminus (\{1\} \times \{1\})$ and $K_1 = K \setminus K_0$. Then, combining the fact that $F_p$ is continuous with \eqref{F_p fund} and \eqref{alpha H exp3} gives that
    \begin{align} \label{alpha H exp4}
        \alpha(\norm{\cdot}_2^p) &= \min \left\{\inf_{(u, r) \in K_0} \{F_p(u, r)\}, \inf_{(u, r) \in K_1}\{ F_p(u, r)\} \right\} \nonumber \\
        &= \min \left\{\inf_{(u, r) \in K_0}\{ F_p(u, r)\}, \inf_{(u, r) \in K_0} \left\{F_p(u, r)^{-1}\right\} \right\} \nonumber \\
        &= \min \left\{\inf_{(u, r) \in K_0} \{F_p(u, r)\}, \left(\sup_{(u, r) \in K_0} \{F_p(u, r)\}\right)^{-1} \right\}.
    \end{align} 
    To further simplify \eqref{alpha H exp4}, we will now show that 
    \begin{equation} \label{subclaim}
        F_p(u, r) \geq F_p(u, s) \text{ holds for all } u \in (0, 1) \text{ and all } -1 \leq r \leq s \leq 1. 
    \end{equation} To that end, suppose $-1 \leq r < s \leq 1$. Define the functions $\phi, \psi : (0, 1) \to \R$ by $\phi(u) \coloneqq u^p + p - 1$ and $\psi(u) \coloneqq (p-1)u^p + 1$ for all $u \in (0, 1)$. Then, 
    \begin{align*}
        F_p(u, r) - F_p(u, s) &= \frac{\phi(u) - pru}{\psi(u) - pru^{p-1}} - \frac{\phi(u) - psu}{\psi(u) - psu^{p-1}} \\ 
        &=\frac{(\phi(u) - pru)(\psi(u) - psu^{p-1}) - (\phi(u) - psu)(\psi(u) - pru^{p-1})}{(\psi(u) - pru^{p-1})(\psi(u) - psu^{p-1})}\\
        &=  \underbrace{\frac{p(s-r)u}{(\psi(u) - pru^{p-1})(\psi(u) - psu^{p-1})}}_{>0}(\psi(u) - \phi(u)u^{p-2})  
    \end{align*} holds for all $u \in (0, 1)$. Therefore, \eqref{subclaim} is implied by:
    \begin{equation}
        \psi(u) - \phi(u)u^{p-2} \geq 0  \text{ holds for all } u \in (0, 1).\label{eq::dependence gone}
    \end{equation}
    Define the functions $G_{1,p}, G_{2,p} : (0, 1) \to \R$ by 
    \begin{align}
        G_{1,p}(u) &\coloneqq \psi(u) - \phi(u)u^{p-2}, \label{eq:: g1 definition}\\
        G_{2,p}(u) &\coloneqq -2u^p + pu^2 - (p-2), \label{eq:: g2 definition}
    \end{align} for all $u \in (0, 1)$. It is immediate to verify that that $G_{1,p}$ and $G_{2,p}$ are differentiable, and that
    \begin{align}
        G_{1,p}'(u) &= (p-1)u^{p-3}G_{2,p}(u), \label{eq:: verify 2} \\
        G_{2,p}'(u) &=2pu(1 - u^{p-2}), \label{eq:: verify 3} \\
        \lim_{u \downarrow 0} G_{1,p}(u) = 1, &\quad \lim_{u \uparrow 1} G_{1,p}(u) = 0, \label{eq:: verify 4} \\
         \lim_{u \downarrow 0} G_{2,p}(u) = -(p-2) < 0, &\quad \lim_{u \uparrow 1} G_{2,p}(u) = 0,\label{eq:: verify 5}
    \end{align} 
    hold for all $u \in (0, 1)$. By \eqref{eq:: verify 3}, $G_{2,p}'(u) \neq 0$ for all $u \in (0, 1)$. Therefore, by \eqref{eq:: verify 5} and the mean value theorem applied to $G_{2,p}$, we have that $G_{2,p}(u) < 0$ for all $u \in (0, 1)$. Therefore, by \eqref{eq:: verify 2}, \eqref{eq:: verify 4} and the mean value theorem applied to $G_{1,p}$, we have that $G_{1,p}(u) > 0$ for all $u \in (0, 1)$. Therefore, \eqref{eq::dependence gone} holds, i.e., \eqref{subclaim} holds. 
    
    Now let us return to simplifying \eqref{alpha H exp4}. Note that $F_p(0, r) = p-1$ for all $r \in [-1, 1]$ and $F_p(1, r) = 1$ for all $r \in [-1, 1)$. By \cref{prop::perfect_symmetry}, we have that $p-1 \geq 1 > \alpha(\norm{\cdot}_2^p)$. Combining this with \eqref{subclaim} and that $F_p$ is continuous on $K_0$, we can simplify \eqref{alpha H exp4} into
    \begin{equation} \label{alpha H exp5}
        \alpha(\norm{\cdot}_2^p) = \min \left\{\inf_{u \in [0, 1)} \{F_p(u, 1)\}, \left(\sup_{u \in [0, 1]} \{F_p(u, -1)\}\right)^{-1} \right\}.
    \end{equation} 
    The function $f_p$ from \eqref{definition of f_p} satisfies \begin{equation}\label{eq::F_p reduction to f_p} f_p(u) = \begin{cases}
F_p(|u|, 1) &\text{ if } u \in [0, 1),\\
F_p(|u|, -1) &\text{ if } u \in [-1, 0).
\end{cases}\end{equation} Since $f_p(-1) = f_p(1) = 1$, \itemcref{technical lemma}{technical lemma i} implies that $f_p(u) \geq 1$ for all $u \in [-1, 1]$ and so by the quasiconcavity of $f_p$ on $[-1, 1]$ we get that \begin{equation}\label{eq::simplifying symmetry F_p}\inf_{u \in [0, 1)} \{F_p(u, 1)\} = 1.\end{equation} Therefore, \eqref{alpha H exp5} simplifies to 
    \begin{align*} 
        \alpha(\norm{\cdot}_2^p) &=  \left(\sup_{u \in [0, 1]} \{F_p(u, -1)\}\right)^{-1} \\
        &= \left(\sup_{u \in [-1, 0]} \{F_p(|u|, -1)\}\right)^{-1} \\
        &= \left(\sup_{u \in [-1, 0]} \{f_p(u)\}\right)^{-1} 
        = \alpha(|\cdot|^p),
    \end{align*} where the first equality follows from \eqref{alpha H exp5} and \eqref{eq::simplifying symmetry F_p}, the third equality follows from \eqref{eq::F_p reduction to f_p}, and the last equality follows from \itemcref{technical lemma}{technical lemma i} and \cref{corollary to lemmas}.
\end{proof}

The following \cref{thm::last theorem} generalizes \cref{prp::answering-proposition to question} into this high-dimensional setting, as referenced in \labelcref{item::C_5}.

\begin{thm} \label{thm::last theorem}
    Let $1 < p_1 < p_2 \leq \dots \leq p_{m-1} < p_m$ and let $\lambda_i \in \R_{++}$ and $r_i \in \{2, p_i \}$ for each $i \in \{1, \dots, m\}$. Let $n$ be a positive integer and consider the function $\sum_{i = 1}^m\lambda_i\norm{\cdot}_{r_i}^{p_i} : \R^n \to \R$. Then, $$\alpha\left(\sum_{i = 1}^m\lambda_i\norm{\cdot}_{r_i}^{p_i}\right) = \min\{\alpha(|\cdot|^{p_1}), \alpha(|\cdot|^{p_m})\}.$$
\end{thm}
\begin{proof}
    First, note that the function $\sum_{i = 1}^m\lambda_i\norm{\cdot}_{r_i}^{p_i} : \R^n \to \R$ is Legendre. We have that \begin{align*}
        \alpha\left(\sum_{i = 1}^m\lambda_i\norm{\cdot}_{r_i}^{p_i}\right) &\in \left[\min_{i \in\{1, \dots, m\}} \left\{ \alpha\left(\lambda_i \norm{\cdot}_{r_i}^{p_i}\right)\right\}, \ \min\left\{ \alpha\left(\lambda_1  \norm{\cdot}_{r_1}^{p_1}\right), \alpha\left(\lambda_m  \norm{\cdot}_{r_m}^{p_m}\right)\right\}\right]  \\
        &= \left[\min_{i \in\{1, \dots, m\}} \left\{ \alpha\left( \norm{\cdot}_{r_i}^{p_i}\right)\right\}, \ \min\left\{ \alpha\left(  \norm{\cdot}_{r_1}^{p_1}\right), \alpha\left(\norm{\cdot}_{r_m}^{p_m}\right)\right\}\right] \\
        &= \left[\min_{i \in\{1, \dots, m\}} \left\{ \alpha\left( |\cdot|^{p_i}\right)\right\}, \ \min\left\{ \alpha\left(  |\cdot|^{p_1}\right), \alpha\left(|\cdot|^{p_m}\right)\right\}\right] \\
        &=  \biggl\{\min\{ \alpha(|\cdot|^{p_1}),\alpha(|\cdot|^{p_m})\}\biggr\},
    \end{align*} where the inclusion follows from \cref{Summation of Positively Hom} and since $\lambda_i \norm{\cdot}_{q_i}^{p_i}$ is positively homogeneous of degree $p_i$ for all $i \in \{1, \dots, m\}$, the first equality follows from \cref{basic proposition}, the second equality from \cref{p norm simple} and \cref{p norm hard}, and the last equality from \cref{Monotonically Decreasing Symmetry Coefficient}.
\end{proof}

We conclude this section with a result analogous to the asymptotic result in \cref{thm:assymptotic theorem}.

\begin{cor} \label{corassymptotic theorem higher dim}
Let $n$ be a positive integer and consider the function $\norm{\cdot}_2^p : \R^n \to \R$. Then, 
\begin{enumerate}[label=(\roman*)]
    \item  $\alpha(\norm{\cdot}_2^p) = \alpha(\norm{\cdot}_p^p) > (2p)^{-1}$ holds for all $p \in (2, \infty)$. \label{itm:assymptotic lower bound cor}
    \item  $\alpha(\norm{\cdot}_2^p) = \alpha(\norm{\cdot}_p^p) \sim(2p)^{-1}$ as $p \to \infty$. \label{itm:assymptotic equivalence cor}
\end{enumerate}
\end{cor}
\begin{proof}
    Combine \cref{p norm simple}, \cref{p norm hard} and \cref{thm:assymptotic theorem}.
\end{proof}
\begin{appendices}
\section[]{Closed form expressions of $\alpha(|\cdot|^p)$ for even positive integers $p$} \label{sec:closed form}
In this section, we compute closed form expressions of $\alpha(|\cdot|)^p$ for some even positive integers $p$. Let $p \geq 4$ be an even integer and recall the polynomial $$g_p(u) \coloneqq u^{2(p-1)} - (p-1)^2 u^p - 2p(p-2)u^{p-1} - (p-1)^2 u^{p-2} + 1$$ over $[0, 1]$ as defined in \eqref{defintion of g_p}. Since $p$ is an integer, we can and will, in this section, extend $g_p$ as a polynomial over $\R$. By \itemcref{technical lemma}{technical lemma i}, finding a closed form expression of $\alpha(|\cdot|^p)$ is equivalent to finding a closed form expression of the unique root of $g_p$ in the interval $[0, 1]$. 
Despite the polynomial being of degree $2(p-1)$, \cref{prop:appendix 1} and \cref{prop:appendix 2} reveal that finding this root is tractable not only for $p=4$, but also for $p \in \{6, 8, 10\}$.

Note that $g_p$ is a \textit{palindromic} polynomial, i.e., its sequence of coefficients $$(1, 0, \dots, 0, -(p-1)^2, -2p(p-2), -(p-1)^2, 0, \dots, 0, 1)$$ forms a palindrome. This observation allows us to make use of \emph{tetrahedral numbers} in our results, where the $n$th tetrahedral number is defined as
$$Te_{n} \coloneqq \sum_{k = 1}^n T_k,$$
and $T_k$ is the $k$th \emph{triangular number}: 
$$T_k \coloneqq \frac{k(k+1)}{2}.$$
For $n \leq 0$, we define $Te_{n} = T_n = 0$ and we note that the following identities hold for all integers $n \geq 1$:
\begin{align}
    T_n &= T_{n-1} + n,  \label{eq:triangular identity} \\
    Te_n &= Te_{n-1} + T_n \label{eq:tetrahedral identity}.
\end{align} 

\begin{prp} \label{prop:appendix 1}
    Let $p \geq 4$ be a positive even integer, then
    \begin{equation}\label{eq:appendix equation 1}
        g_p(u) = (u+1)^4\sum_{k = 0}^{2(p-3)}Te_{\min\{k, 2(p-3)-k\}+1}(-1)^ku^k
    \end{equation} holds for all $u \in \R$.
\end{prp}
\begin{proof}
Fix $p \geq 4$. Let 
$$(u+1)^4\sum_{k = 0}^{2(p-3)}Te_{\min\{k, 2(p-3)-k\}+1}(-1)^ku^k = \sum_{i = 0}^{2(p-1)}c_i u^i $$ 
hold for all $u \in \R$ and some coefficients $c_i \in \R$. First, we note that $c_0 = 1$. For any $i \in \{1, \dots, p-3\}$, we get that 
$$c_i = (-1)^i(Te_{i+1} - 4 Te_{i} + 6Te_{i-1} - 4Te_{i-2} + Te_{i-3}),$$ 
which after an evaluation gives that $c_i = 0$ for $i \in \{1, \dots, p-3\}$, since in general  $$Te_{n} - 4 Te_{n-1} + 6Te_{n-2} - 4Te_{n-3} + Te_{n-4} = 0$$ holds for any $n \in \mathbb{Z} \setminus \{1\}$. Moreover, it is straightforward to verify that \begin{equation*}
    c_{p-2} = (-1)^{p-2}(Te_{p-3} - 4Te_{p-2} + 6 Te_{p-3} - 4Te_{p-4} + Te_{p-5}) = -(p-1)^2 \\
\end{equation*} and \begin{equation*}
    c_{p-1} = (-1)^{p-1}(Te_{p-4} - 4Te_{p-3} + 6 Te_{p-2} - 4Te_{p-3} + Te_{p-4}) = -2p(p-2),\\
\end{equation*} 
where we here used twice that $p$ is even. Since the right hand side of \eqref{eq:appendix equation 1} is a product of two palindromic polynomials, we have that $\sum_{i = 0}^{2(p-1)}c_i u^i$ is also a palindromic polynomial, i.e., $c_i = c_{2(p-1)-i}$ must hold for all $i \in \{0, \dots, 2(p-1)\}$, and so \eqref{eq:appendix equation 1} holds.
\end{proof}

It will at this point be convenient to define, for any even integer $p \geq 4$, the palindromic polynomials $$h_{1, p}(u) \coloneqq \sum_{k = 0}^{p-4} T_{\min\{k, p-4-k\} + 1} (-1)^k u^k$$ and $$h_{2, p}(u) \coloneqq (-1)^{\frac{p}{2}-1}(p-1)u^{\frac{p}{2}-1} + \sum_{i = 0}^{p-2}(-1)^iu^i$$ for all $u \in \R$.

\begin{prp}\label{prop:appendix 2}
    Let $p \geq 4$ be a positive even integer, then
    \begin{equation}\label{eq:appendix equation 2}
        h_{1, p}(u) h_{2, p}(u) = \sum_{k = 0}^{2(p-3)}Te_{\min\{k, 2(p-3)-k\}+1}(-1)^ku^k
    \end{equation} holds for all $u \in \R$.
\end{prp}
\begin{proof}
    Suppose that $\sum_{i = 0}^{2(p-3)}c_i u^i \coloneqq h_{1, p}(u) h_{2, p}(u)$ holds for all $u \in \R$ and some coefficients $c_i \in \R$. We have that \begin{align}
        h_{1, p}(u)h_{2, p}(u) &= \sum_{i = 0}^{p-2}\sum_{j = 0}^{p-4}(-1)^{i+j}T_{\min\{j, p-4-j\}+1}u^{i+j} \label{eq:first one}\\
        &\ + (p-1)\sum_{j = 0}^{p-4}(-1)^{\tfrac{p}{2}-1+j}T_{\min\{j, p-4-j\}+1} u^{\tfrac{p}{2}-1 + j}. \label{eq:second one}
    \end{align} 
    Let $k \in \{0, \dots, \tfrac{p}{2}-2\}$. Then \eqref{eq:first one}, but not \eqref{eq:second one}, will contribute to $c_k$. Since, $\min\{k, p-4-k\} = k$, we get that
    $$c_k = (-1)^k\sum_{j = 0}^k T_{j+1} = (-1)^k Te_{k+1}.$$

    Let us now cover the case when $k \in \{\tfrac{p}{2}-1, \dots, p-3\}$, using induction over $k$. Let us fix $k = \tfrac{p}{2}-1$. When expanding, there will be contributions from both the right hand side of \eqref{eq:first one} and \eqref{eq:second one} to $c_k$. From \eqref{eq:first one} we get the terms \begin{equation}\label{eq:lasteq}
        (-1)^k\left(\sum_{j = 0}^{\tfrac{p}{2}-2} T_{j+1} + T_{p-4-k+1} \right)
    \end{equation} and from \eqref{eq:second one} we get $(-1)^k(p-1)T_1$, which when added to \eqref{eq:lasteq} yield
    \begin{align*}
        c_{k} &= (-1)^k\left( \sum_{j = 0}^{\tfrac{p}{2}-2} T_{j+1} + T_{\tfrac{p}{2} - 2} + (p-1) T_1\right) \\
        &= (-1)^k\left( Te_{\tfrac{p}{2}-1} + T_{\tfrac{p}{2} - 1} - (\tfrac{p}{2}-1) + (p-1)\right) \\
        &= (-1)^k\left( Te_{\tfrac{p}{2}-1} + T_{\tfrac{p}{2} - 1} + \tfrac{p}{2}\right) \\
        &= (-1)^kTe_{k+1},
    \end{align*} where the second and final equalities use \eqref{eq:triangular identity} and \eqref{eq:tetrahedral identity}.

    Let us now suppose that the induction hypothesis $c_{k-1} = (-1)^{k-1}Te_k$ holds for some $k \in \{\tfrac{p}{2}, \dots, p-3\}$. We will now show that $c_{k} = (-1)^{k}Te_{k+1}$. Similarly expanding the right hand side of \eqref{eq:first one} and \eqref{eq:second one} we get
    \begin{align*}
        c_k &= (-1)^k\left( \sum_{j = 0}^{\tfrac{p}{2}-2} T_{j+1} + \sum_{j = \tfrac{p}{2}-1}^k T_{p-3-j} + (p-1)T_{k-\left(\tfrac{p}{2}-1 \right) + 1} \right) \\
        &=(-1)^k\left( \sum_{j = 0}^{\tfrac{p}{2}-2} T_{j+1} + \sum_{j = \tfrac{p}{2}-1}^{k-1} T_{p-3-j} + T_{p-3-k} + (p-1)T_{k-\tfrac{p}{2} + 2} \right) \\
        &=(-1)^k\left(  \sum_{j = 0}^{\tfrac{p}{2}-2} T_{j+1} + \sum_{j = \tfrac{p}{2}-1}^{k-1} T_{p-3-j} + T_{p-3-k} + (p-1)\left(T_{(k-1) - \tfrac{p}{2} + 2} + k -\tfrac{p}{2} + 2 \right)\right) \\
        &=(-1)^k\left(Te_k + T_{p-3-k} + (p-1)\left(k - \tfrac{p}{2} + 2\right)\right), 
    \end{align*} where we in the third equality use \eqref{eq:triangular identity} and in the last equality we use the induction hypothesis:
    \begin{align*}
        c_{k-1} &= (-1)^{k-1}Te_k \\
        \implies Te_k &=  \sum_{j = 0}^{\tfrac{p}{2}-2} T_{j+1} + \sum_{j = \tfrac{p}{2}-1}^{k-1} T_{p-3-j} +  (p-1)T_{(k-1) - (\tfrac{p}{2}-1) + 1}.
    \end{align*}
    
    Finally, \begin{align*}
        T_{p-3-k} + (p-1)(k-\tfrac{p}{2} + 3) &= \frac{(p-3-k)(p-2-k)}{2} + \frac{(p-1)(2k - p + 4)}{2} \\
        &= \frac{(k+1)(k+2)}{2} = T_{k+1}
    \end{align*} and so $c_k = (-1)^k\left(Te_k + T_{k+1}\right) = (-1)^kTe_{k+1}.$ By induction, $c_k = (-1)^kTe_{k+1}$ holds for all $k \in \{\tfrac{p}{2}-1, \dots, p-3\}$.
    
    Since both $h_{1, p}$ and $h_{2, p}$ are palindromic polynomials, the resulting polynomial $h_{1, p} h_{2, p}$ must also be palindromic. Moreover, the right hand side of \eqref{eq:appendix equation 2} is a palindromic polynomial, implying that the result that $c_k = (-1)^k Te_{k+1}$ for $k \in \{0, \dots, p-3\}$ along with the palindromic property shows the equality \eqref{eq:appendix equation 2}.
\end{proof}

Combining \cref{prop:appendix 1} and \cref{prop:appendix 2} we have reduced the problem of finding a closed form expression of $\alpha(|\cdot|^p)$, for even integers $p \geq 4$, to finding a root of the polynomials $h_{1, p}$ or $h_{2, p}$ with degrees $p-4$ and $p-2$, respectively. This problem may seem intractable when $p-2 \geq 5$, i.e., when $p \geq 8$ (since $p$ is an even integer). However, we will further explore the palindromic properties of $h_{1,p}$ and $h_{2,p}$, allowing us to handle the cases with $p = 8$ and $p=10$.

The study of palindromic polynomials, also called \textit{reciprocal} polynomials, has a rich history. See for instance \cite{durand1961}, where it was shown that for palindromic polynomials $h$ of even degree $d$, there exists a polynomial $\tilde{h}$ of degree $d/2$ such that $h(u) = u^{d/2} \tilde{h}(u + u^{-1})$ holds for all $u \neq 0$. Since both of our polynomials $h_{1, p}$ and $h_{2, p}$ are palindromic of even degrees, we apply this method to reduce the problem of finding a zero to $h_{1,p}$ and $h_{2,p}$ to the one of finding a zero of the polynomials $\tilde{h}_{1, p}$ and $\tilde{h}_{2, p}$ of degree $(p-4)/2$ and $(p-2)/2$, respectively. 
Since roots to polynomials up to degree 4 can be computed in closed form, the problem becomes tractable for $p = 8$ and $p = 10$. The case of $p=12$ seems to be intractable, since the positive real roots we seek are the real roots of $h_{2, 12}$, which has degree $5$.

\begin{exmp}

Based on these results, let us illustrate how to find $\alpha(|\cdot|^p)$ for $p = 8$. For each $u \neq 0$, we have
\begin{align*}
    h_{2, 8}(u) &= u^6 - u ^5 + u^4 - 8u^3 + u^2 - u + 1 \\
    &= u^3((u+u^{-1})^3 - (u+u^{-1})^2 - 2(u+u^{-1}) - 6),
\end{align*}
giving us the equation
$$\tilde{h}_{2, 8}(v) = v^3 - v^2 - 2v - 6 = 0,$$
which has the real and positive solution 
$$v_0 \coloneqq \frac{1}{3} \left(1 + \sqrt[3]{7 (13 - 9 \sqrt{2})} + \sqrt[3]{7 (13 + 9 \sqrt{2})}\right) \approx 2.63.$$ 
Since the range of the mapping $u \mapsto u+u^{-1}$, for $u \in (0, 1]$, equals $[2, \infty)$, there exists a solution $u_0 \in (0, 1]$ to the equation $u_0 + u_0^{-1} = v_0.$ Like in \cref{alg:compute symmetry for 1d norm}, we can then extract a closed form expression of $\alpha(|\cdot|^8)$ after evaluating $f_8(-u_0)^{-1}$. This leads us to the monstrous expression \begin{align}\alpha(|\cdot|^8) &=\frac{13996800 + 2239488\left(\beta_2 + \sqrt[3]{7 (13 - 9 \sqrt{2})} + \sqrt[3]{7 (13 + 9 \sqrt{2})}\right) + \beta_3^{8}}{1679616 + 48 \beta_3^{7} + 7 \beta_3^{8}} \approx 0.0982  \label{eq:alpha8}
\end{align} where \begin{align*} 
\beta_1 &\coloneqq 1 + \sqrt[3]{7} \left(\sqrt[3]{13 - 9 \sqrt{2}} + \sqrt[3]{13 + 9 \sqrt{2}}\right)\\
\beta_2 &\coloneqq \sqrt{ \beta_1^{2} -36} \nonumber \\
\beta_3 &\coloneqq \beta_1 + \beta_2 \nonumber. \end{align*}

\end{exmp}
\begin{exmp}
An analogous procedure lets us compute the closed form analytical value of $\alpha(|\cdot|^{10})$ as \begin{align}\alpha(|\cdot|^{10}) &=  \frac{201 + 23 \cdot 3^{\frac{2}{3}} + 126 \cdot \sqrt[3]{3} -\left( 135 + 50 \cdot 3^{\frac{2}{3}} + 45 \cdot \sqrt[3]{3} \right)\sqrt{3^{\frac{2}{3}} + 2 \cdot \sqrt[3]{3}-3}}{2 \left(3^{\frac{2}{3}} + 12\left(1 + \sqrt[3]{3}\right)\right)} \approx 0.0733 \label{eq:alpha10}
\end{align} and of $\alpha(|\cdot|^6)$ as presented in \cref{tab::special_cases_of_symmetry_coefficient}.
\end{exmp}
\begin{rem}
    The symmetry coefficient $\alpha(|\cdot|^3)$ can be evaluated in its closed form expression by finding the real positive roots of the $4$th degree polynomial given in \eqref{defintion of g_p}. On the other hand, an analogous method for finding closed form expressions for other \textit{odd} values $p$ does not seem straightforward, nor does it seem like the methods conveyed here can be extended to this setting.
\end{rem}
\begin{rem}
The second polynomial $h_{2, p}$ has been studied in literature, since for some even values of $p$, it equals a \textit{cyclotomic} polynomial, with its middle coefficient being \textit{perturbed}, see for example \cite{dresden2004middle,mossinghoff1998perturbing}. There might be some way to exploit the structure---unknown to the authors---in order to analytically determine the roots of $h_{2, p}$ for even higher values of even integers $p$. 
\end{rem}

\end{appendices}
\section*{Acknowledgments}
M. Nilsson and P. Giselsson acknowledge support from the ELLIIT Strategic Research Area and the Wallenberg AI, Autonomous Systems, and Software Program (WASP), funded by the Knut and Alice Wallenberg Foundation. Additionally, P. Giselsson acknowledges support from the Swedish Research Council.
\printbibliography
\end{document}